\documentclass[oneside,reqno]{amsart}
\usepackage[T1]{fontenc}		
\usepackage{ankita-packages}		

\allowdisplaybreaks

\usepackage{ankita-thms}			
\usepackage{ankita-spacing}
\usepackage{ankita-commands}		
\newcommand{\ur}[1]{\textcolor{blue}{#1}}
\newcommand{\an}[1]{\textcolor{red}{#1}}

\makeatletter

\def\@theoremstyleplain#1#2#3#4{%
    \setlength{\topsep}{2cm} 
    \@theoremstyle@inner{#1}{#2}{#3}{#4}%
}

\def\@theoremstyledescription#1#2#3#4{%
    \setlength{\topsep}{6pt} 
    \@theoremstyle@inner{#1}{#2}{#3}{#4}%
}

\def\@theoremstyleremark#1#2#3#4{%
    \setlength{\topsep}{6pt} 
    \@theoremstyle@inner{#1}{#2}{#3}{#4}%
}

\makeatother

\makeatletter
\@namedef{subjclassname@2020}{\textup{2020} Mathematics Subject Classification}
\makeatother

\title[Temperatures of Robin Hood]%
{Temperatures of Robin Hood}

\author[Ankita Dargad]{Ankita Dargad}
\address{Department of Mathematics, Indian Institute of Technology Bombay, Powai, Mumbai 400076}
\email{ankitadargad.iitb@gmail.com}

\author[U. Larsson]{Urban Larsson}
\address{Department of Industrial Engineering and Operations Research, Indian Institute of Technology Bombay, Powai, Mumbai 400076}
\email{larsson@iitb.ac.in}

\author[N. Balachandran]{Niranjan Balachandran}
\address{Department of Mathematics, Indian Institute of Technology Bombay, Powai, Mumbai 400076}
\email{niranj@math.iitb.ac.in}


\date{\today}

\begin{document}

\begin{abstract}
Cumulative Games were introduced by Larsson, Meir, and Zick (2020) to bridge some conceptual and technical gaps between Combinatorial Game Theory (CGT) and  Economic Game Theory. The partizan ruleset {\sc Robin Hood} is an instance of a Cumulative Game, viz., {\sc Wealth Nim}. It is played on multiple heaps, each associated with a pair of cumulations, interpreted here as wealth. Each player chooses one of the heaps, removes tokens from that heap not exceeding their own wealth, while simultaneously diminishing the other player's wealth by the same amount. In CGT, the {\em temperature} of a {\em disjunctive sum} game component is an estimate of the urgency of moving first in that component. It turns out that most of the positions of {\sc Robin Hood} are {\em hot}. The temperature of {\sc Robin Hood} on a single large heap shows a dichotomy in behavior depending on the ratio of the wealths of the players. Interestingly, this bifurcation is related to Pingala (Fibonacci) sequences and the Golden Ratio $\phi$: when the ratio of the wealths lies in the interval $(\phi^{-1},\phi)$, the temperature increases linearly with the heap size, and otherwise it remains constant, and the mean values has a reciprocal property. It turns out that despite {\sc Robin Hood} displaying high temperatures, playing in the hottest component might be a sub-optimal strategy.
\end{abstract}
\maketitle

\section{Introduction}
In the Era of Pingala, two wetland tribes engage in a dispute over land pieces on various islands for farming, which escalates into a war. Being honorable tribes, their chiefs agreed to certain rules for fighting: as a preparation, each tribe allocates a number of soldiers to each island, and the tribe with the smaller number of soldiers gets to start. 
On day one, this tribe gets to challenge their opponent on an island of their choice, and by the principle of ``first player advantage'' they get to weaken the opponent's strength on that island. The next day, the other tribe retaliates and attacks any island of their choice. This continues, and, while fighting, every day some land pieces gets ruined. The war concludes when only one tribe is left on each island. 
At that point, they sum up the total remaining fertile land pieces on their respective conquered islands, and return home to respective villages to celebrate the end of the war. 
\footnote{This has some resemblance with the classical ``Colonel Blotto'' war game \cite{EB1921}. The main difference is that they use simultaneous play, and their game concerns the assignment of the forces to the islands, while in our setting, we will regard the assignments of soldiers as given, and the main challenge is the sequential selections for `the next fight'. The story is 100\% fictional.}


How can the tribes maximize their gains in terms of conquered land pieces? And, which island should they target first?

\begin{figure}[ht]
    \centering
    \begin{tikzpicture}
        \pic  at (0,0) {position1};
        \pic at (5.5,0) {position2};
        \draw[->,black,thick,red] (2.2,0)--(3.26,0); 
        
        \pic at (5.5,-4.2) {position3};
        \draw[->,black,thick,blue] (5.5,-1.69)--(5.5,-2.55); 
        \pic at (0,-4.2) {position4};
        \draw[->,black,thick,red] (3.3,-4.2)--(2.2,-4.2); 
    \end{tikzpicture}
    \caption{A three days' war started by the red tribe. The war ended on the 3rd day as only one tribe is left on each island. Ruined land pieces and beaten soldiers are colored gray. The colored arrows specify the attacking tribe.}
    \label{fig:3 days War in the village}
\end{figure}
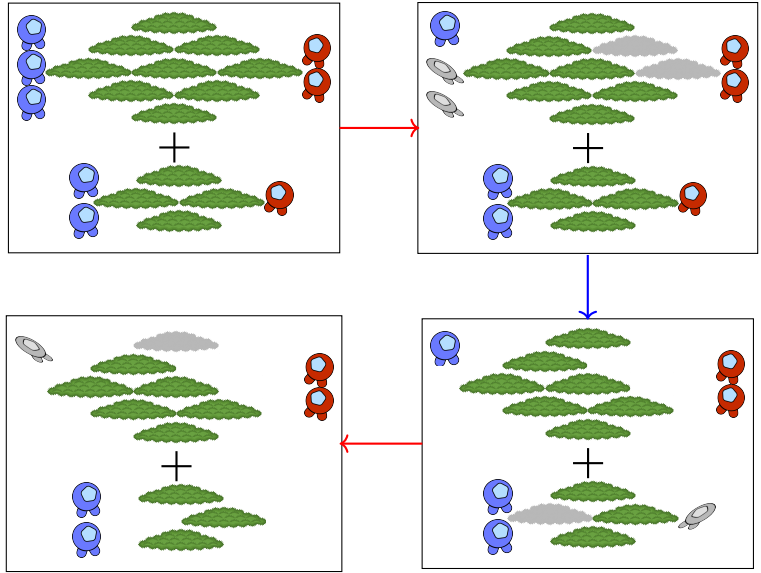

This scenario describes a game of a ruleset dubbed {\sc Robin Hood}, which is an instance of {\sc Wealth Nim}, a {\sc Cumulative Game} \cite{URYCG2020}. This is a two-player alternating play combinatorial game \cites{S2013,BCG2004} that is played on a finite number of heaps of finite sizes. To each heap, each player has an associated wealth, an integer that we shall refer to as the {\em heap wealth}, which determines their strength on that heap. On their turn, a player 
\begin{itemize}
    \item chooses one of the heaps,
    \item removes a certain number of tokens from that heap not exceeding their own heap wealth, and
    \item reduces the opponent's wealth on that heap by the same amount.
\end{itemize}
 
 A player who cannot move loses (normal play). Let $\Nat=\{1,2,\ldots\}$, and let $\Nat_0=\Nat\cup\{0\}$. We denote by $(n;\, a,b)$ a {\sc Robin Hood} instance comprising of a single heap of size $n\in \Nat_0$, Left's wealth equal to $a\in \Nat_0$ and Right's wealth equal to $b\in \Nat_0$. In case the player with greater wealth removes more tokens than the opponent's wealth, their wealth drops below zero. If so, by convention, we set the opponent's wealth to zero, since they cannot play further. 
In the scenario described above, each island represents a heap, the number of land pieces represents the heap size and the tribes' strengths represent the players' wealths. The blue (resp. red) tribe represents Left (resp. Right) player.\footnote{Rumors, from the Sherwood forest, tell that Robin Hood got inspired by folktales about the wetland tribes' wars, in the spirit: ``I take, and you pay!''.}

Figure~\ref{fig:3 days War in the village} illustrates a play on a {\sc Robin Hood} instance with two heaps of sizes 9 and 4 (land pieces). Left (blue tribe) has heap wealths of 3 and 2, and Right (red tribe) has heap wealths of 2 and 1, respectively. The `$+$' sign denotes the disjunctive sum, meaning that on their turn, a player chooses exactly one heap (island) to play on, while the other heap remains the same. The war ends when only one tribe remains on each island, but {\sc Robin Hood} continues due to the normal play convention. 
It ends when the current player has no token to remove. In this example, after the war, Left has at most three rounds of play before she runs out of tokens, while Right's resources can still be abundant. Thus, Right wins the game. (By looking at the fine details, Left felt somewhat disappointed as she realized that if she instead would have kept responding in the first component, then Right would have won by only two tokens.)

By separating an instance of {\sc Robin Hood} into a war phase, and a conclusion phase, we make a first observation. After the war phase, for each player there is a total number of remaining tokens in all heaps, where they still have non-zero heap wealth. This number can be considered as `their number of free moves', as those tokens can be removed one at a time, without interference from their opponent. 

Intuitively, both players desire to start as they get to diminish the opponent's strength as early as possible by as much as possible. Additionally, it turns out that, in the Figure~\ref{fig:3 days War in the village} example, playing in the first component is preferential for either player.

 To elucidate: If Right starts in the first component, then as discussed before, Left can respond in the first component, and Right can win by $(5-3)=2$ moves,  whereas, if Right starts in the second component, Left can respond in the first component and gain 7 moves, while Right receives only 2 moves from the second component. Therefore, Right loses by $(7-2)=5$ moves. Thus, the total loss for Right by starting in the second component instead of the first one is $2-(-5)= 7$ moves. 

If Left starts and plays in the first component, she gains 7 land pieces in this component and also receives 2 land pieces from the second one, leading to a win by 9 moves. However, if Left starts in the second component, Right can play in the first component on his turn. In the end, Left gains 3 land pieces from the second component, while Right gets 5 land pieces from the first component. Therefore, Left loses by 2 moves. Thus, the total loss for Left by starting in the second component instead of the first is $9-(-2)=11$.

The notion of {\it temperature} \cites{S2013,BCG2004}  (here Definition~\ref{def:temp}) is an estimate,  which attempts to capture this sense of urgency. In many cases, a higher temperature indicates greater urgency, while a lower temperature indicates less urgency. In keeping with this terminology, a game with positive temperature is called a ``hot'' game. We will discuss this in greater detail in Section~\ref{sec: basicsetup}. Another related concept is the {\em mean value} of a game. It is an estimate of how good a game is for the respective player; a larger mean value is usually better for Left and vice versa.

Based on this  intuition 
our main question is:
\vspace{3mm}

\begingroup
    \centering
    \textit{What are the temperatures and mean values of {\sc Robin Hood}?}\\
\endgroup
\vspace{3mm}

CGsuite \cite{CGSuite} guides us on the temperatures and the mean values of the two {\sc Robin Hood} games, $(n;5,4)$ and $(n;5,3)$, with varying $n$. For any game $G$, let $\te(G)$ and $\m(G)$ denote the temperature and mean of $G$, respectively. See Table~\ref{tab:tempwithmean}.

\renewcommand{\arraystretch}{1}%
\begin{table}[ht]
    \centering
    \begin{tabular}{|c|c|c|c|c|}
        \hline
        $ n$ & $\te(n;5,4)$&$\m(n;5,4)$& $\te(n;5,3)$& $\m(n;5,3)$\\ \hline
        \hline
         $3$  & $0$ &0& $0$ & 0\\
        \hline
         $4$  & $0$ &0& $1/2$ & 1/2\\
         \hline
         $5$  & $1/2$ &1/2& $1$ &1\\
         \hline
         $6$  & $5/4$ &3/4& $7/4$ &5/4\\
         \hline
         $7$  & $17/8$ &7/8& $19/8$ &13/8\\
         \hline
         $8$  & $49/16$ &15/16& $3$ &2\\
         \hline
         $9$  & $\pmb 4$ &\it1& $7/2$ &5/2\\
         \hline
         $10$ & $\pmb5$ &\it 1& $\pmb4$ &\it 3\\ 
         \hline
         $11$ & $\pmb6$ &\it 1& $\pmb4$ &\it 4\\
         \hline
         $12$ & $\pmb7$ &\it 1& $\pmb4$ &\it 5\\
         \hline
         $13$ & $\pmb8$ &\it 1& $\pmb4$ &\it 6\\
         \hline
         
    \end{tabular}
    \caption{The temperatures and mean values of the games $(n;5,4)$ and $(n;5,3)$ for a few initial heap sizes $n$. We indicate in italicized and bold when patterns emerge.}
    \label{tab:tempwithmean}
\end{table}


Consistency can be observed in the values of the temperature and the mean of the two games in Table~\ref{tab:tempwithmean} for large heap sizes. 
It also shows a reciprocal behavior between the temperature and the mean values. The temperature (mean value) of the game $(n;5,4)$ increases (stabilizes) as the heap size grows. Whereas, the temperature (mean value) of the game $(n;5,3)$ stabilizes (increases) with increasing heap size.

Next, we observe the temperatures for different wealth pairs in Table~\ref{tab: Relation between Temperature and $n,a$ and $b$}. 

\renewcommand{\arraystretch}{1}%
\begin{table}[ht]
    \centering
    \begin{tabular}{|l|l|l|l|}
        \hline
        
        $(n;a,b)$ & \textit{Property} & $\te(n;a,b)$ & \textit{Bound}  \\ 
        \hline
        \hline
        \bm{$(n;1,1)$}   & {\bf Increasing}     & \bm{$n-1$}     & \bm{$ 1$} \\  
        \hline
        $(n;1,2)$   & Stabilizing    & $1$       & $ 3$  \\
        \hline
        $(n;1,3)$   & Stabilizing    & $1$       & $ 4$  \\
        \hline
        \bm{$(n;2,3)$}   & {\bf Increasing }    & \bm{$n-3$}     & \bm{$8$}  \\
        \hline
        $(n;2,4)$   & Stabilizing    & $2$       & $ 6$ \\ 
        \hline
        $(n;2,9)$   & Stabilizing    & $2$       & $ 11$ \\
        \hline
        \bm{$(n;3,4)$ }  & {\bf Increasing}     & \bm{$n-4$ }    & \bm{$7$ } \\
        \hline
        $(n;3,5)$   & Stabilizing    & $4$       & $10$ \\
        \hline
        $(n;3,6)$   & Stabilizing    & $3$       & $9$ \\
        \hline
        \bm{$(n;5,8)$}  & {\bf Increasing}    & \bm{$n-8.5$ }  & \bm{$15$ }\\
        \hline
        $(n;5,9)$   & Stabilizing    & $6$       & $ 16$ \\
        \hline
        \bm{$(n;7,11)$ } & {\bf Increasing }    & \bm{$n-11.5$}  & \bm{$20$}  \\
        \hline
        $(n;7,12)$  & Stabilizing    & $9$       & $15$ \\
        \hline
    \end{tabular}
    \caption{Temperatures of the  games $(n;a,b)$ for different pairs $(a,b)$. The second column specifies the behavior of the temperature with increasing $n$. The last column is the lower bound of $n$ for which the property holds.
    }
    \label{tab: Relation between Temperature and $n,a$ and $b$}
\end{table}

Table~\ref{tab: Relation between Temperature and $n,a$ and $b$} illustrates whether the temperature of $(n;a,b)$ is stabilizing or increasing with $n$ for different pairs $(a,b)$. 

When one tribe is significantly stronger than the other and the number of land pieces is large, the stronger tribe will have a high chance of getting huge loot of the war, regardless of the actions of the weaker tribe. As a result, neither tribe has a strong desire to start the war, irrespective of the number of land pieces. The less desire to start the war shows that the heat in this situation should not increase with the number of land pieces.
However, when the players' wealth is relatively equal, both players benefit from starting the game. This increases their desire to begin, leading to a rise in temperatures as the heap size increases.


This indicates that the two different patterns of the temperature values is related to the ratio of players' wealth. It is clear from Table~\ref{tab: Relation between Temperature and $n,a$ and $b$} that whenever the ratio of wealth (larger to smaller) is bigger than $12/7$, the temperature stabilizes, and it keeps increasing when the ratio is larger than $11/7$. 


Our main result provides a complete description of the temperatures and mean values of {\sc Robin Hood} when played on a single large heap. Let $\phi=\frac{1+\sqrt{5}}{2}$ denote the {\em golden ratio}.
\begin{theorem}[Main Theorem]\label{thm: main theorem}
Let $G = (n;a,b)$ be an instance of {\sc Robin Hood}, where \(n\), \(a\) and \(b\) are 
positive integers.  Let $(U_k)_{k\geq 0}$ be the unique sequence of positive integers such that
\begin{enumerate}
    \item $U_0\geq U_1$,
    \item $U_{k+2}=U_{k+1}+U_k \text{ for all } k\geq 0$, and
    \item for some $q\geq0$, $U_q=\min\Set{a,b}$ and $U_{q+1}=\max\Set{a,b}$. 
\end{enumerate} 
    If $n\geq a+b$, then $G$ is a hot game and for all sufficiently large $n$, the temperature of the game $G$, denoted by $\te(G)$, is\\   $$\te(G)=
  \begin{cases} 
        b-U_0 & \text{ if }~ \frac{a}{b} < \phi^{-1};\\
        n-a + \frac{U_0-b}{2} & \text{ if }~ \phi^{-1}<\frac{a}{b} <1;\\
        n-a & \text{ if }~ \frac{a}{b} = 1;\\
        n-b+\frac{U_0-a}{2} & \text{ if }~ 1<\frac{a}{b} < \phi;\\
        a-U_0 & \text{ if }~ \phi<\frac{a}{b},\\
  \end{cases}$$
  
  and the mean value of the game $G$, denoted by $\m(G)$, is\\   $$\m(G)=
  \begin{cases} 
        -\left(n-(a+b)+U_0\right) & \text{ if }~ \frac{a}{b} < \phi^{-1};\\
        \frac{U_0-b}{2} & \text{ if }~ \phi^{-1}<\frac{a}{b} <1;\\
        0 & \text{ if }~ \frac{a}{b} = 1;\\
        \frac{a-U_0}{2} & \text{ if }~ 1<\frac{a}{b} < \phi;\\
        n-(a+b)+U_0 & \text{ if }~ \phi<\frac{a}{b}.\\
  \end{cases}$$
\end{theorem}
\vspace{1 mm}

{\em Hotstrat} is a playing strategy \cite{S2013} that implies `playing in the hottest component in a disjunctive sum of games is an optimal strategy for either player'.\footnote{Hotstrat also specifies the move to be made in the hottest component \cite{S2013}.} Here, by optimal strategy, we mean the alternating play strategy that maximizes the earning of the players and minimizes the loss in terms of `number of moves'. 
To compute the temperature, a geometric structure, called {\em thermograph} \cite{S2013} (here Definition~\ref{def:thermograph}), is used.
The notions of `urgency' and temperature are not always the same in the sense that playing on the hottest component (hotstrat) is not always an optimal strategy. If a ruleset satisfies hotstrat, then knowing the temperature and the thermograph of a game reveals all the facets of the game. Many {\sc Robin Hood} positions satisfy hotstrat, for instance the position in Figure~\ref{fig:3 days War in the village}. However, {\sc Robin Hood} is even more interesting, as the following theorem shows.
\begin{theorem}[No Hotstrat]\label{thm:nohotstrat}
There exists a {\sc Robin Hood} disjunctive sum game, with components of distinct temperatures, such that the unique winning  move is in the coolest component.
\end{theorem}
We postpone the proof of this result until the end of Section~\ref{sec: organized RH}.



The rest of the paper is organized as follows. 

\begin{itemize} 
    \item In Section~\ref{sec:literature} we mention two papers that inspired this work. 
    \item In Section~\ref{sec: basicsetup} we review the notion of temperature and mean value using Left and Right stops, along with thermographs.
    \item In Section~\ref{sec: organized RH} we present some overarching facts about {\sc Robin Hood}, and we provide an example where hotstrat fails.
    \item In Section~\ref{sec:RH meets Little John}, we see the connection between Pingala sequences and {\sc Robin hood}.
    \item  Section~\ref{sec:ping} collects some relevant facts about Pingala sequences and its generalizations. 
    \item In Section~\ref{sec:LJ}, we study a simpler game that we dub {\sc Little John}.
    \item In Section~\ref{sec:stops} we show that for large heaps, the Left and Right stops for {\sc Robin Hood} and {\sc Little John} are the same, and we also study the thermographs of {\sc Little John}.
    \item In Section~\ref{sec:main}, we prove the main result, Theorem~\ref{thm: main theorem}, by first justifying that the thermographs of {\sc Little John} and {\sc Robin Hood} are the same for large heap sizes.
\end{itemize}

\section{Literature review}\label{sec:literature}
Our ruleset has an, at first surprising but after a while fairly obvious, resemblance of the impartial ruleset {\sc Euclid}, which is played on two non-empty heaps of pebbles. A player must remove a multiple of the size of the smaller heap from the larger heap. A position is represented by a pair of positive integers $(x,y)$, where say $x\le y$. Note that if $x=y$, then the position is terminal. Example: $(2,7)\rightarrow (2,3)\rightarrow (1,2)\rightarrow (1,1)$. Since we put the requirement that (both) heaps remain non-empty, then no more move is possible. Note that the losing moves are forced. 
 
Optimal play reduces to minimizing the relative distance of the heaps. 

\begin{theorem}[\cite{cole1969game}]\label{thm:euclid}
A player wins {\sc Euclid} if and only if they can remove a multiple of the smaller heap such that the ratio of the heap sizes $(x,y)$, satisfies $1\le y/x<\phi$. 
\end{theorem}


In {\sc Robin Hood}, the player with lesser wealth cannot remove a non-trivial multiple of their own wealth. Similarly, in {\sc Euclid}, for positions $(x,y)$ where $y/x<\phi$, players are restricted from removing non-trivial multiples. Consequently, the removal recurrences are identical in both scenarios.


 


Cumulative Games were defined in a broad sense in the preprint \cite{URYCG2020}. The purpose of that monograph is to explore an intersection of classical game theory with combinatorial game theory. The ruleset {\sc Wealth Nim} a.k.a. {\sc Wealth Pebbles} is introduced as an example where player cumulations are part of the rules of how to move, but do not contribute to the payoffs that the players gain when the game ends. Since the purpose of \cite{URYCG2020} is to provide a broad framework for further study, no efforts were made to solve proposed individual games and rulesets. The current paper is among the first ones to do so.


\section{Some basics}\label{sec: basicsetup}

In order to establish the foundation for proving our main results regarding temperatures and mean values, we revisit the concept of a `thermograph'. Understanding thermographs requires familiarity with several key terminologies from CGT. We will not discuss standard CGT concepts, such as game comparison and canonical forms, while they are less critical in this context, and these topics are well-covered in the existing literature (e.g., \cites{BCG2004,S2013}). We begin by defining the pivotal concept of a Number game. 

Intuitively, a game is called a `Number' if each player prefers that the other player starts. 
Recall that a {\em sub-position} of a game can be the game itself or any option of the game or any option of options, etc.

\begin{definition}[Numbers]\label{def: Number}
     A short game $x$ is a Number if, in the canonical form of $x$, every sub-position $y$ satisfies $y^L<y^R$ for all $y^L$ and $y^R$.\footnote{The definition of a Number in \cite{S2013} does not hold for $\{*\mid*\}$ and many other literal form games that equal some canonical form Number.}
\end{definition}

As usual, $\varnothing$ denotes the empty set (of options). Every game is associated with a Number game in the following sense. 

\begin{definition}[Mean Value]
Consider a short game $G$. Its mean value $\m(G)$ is the Number such that, for any positive dyadic $\epsilon$, for all sufficiently large $n$, $n\cdot m(G)-\epsilon\le n\cdot G\le n\cdot m(G)+\epsilon$.
\end{definition}
In \cite{S2013}*{Theorem 3.23} it is proved that every game has a mean value.

The most basic Number games are as follows: For all \( k \in \mathbb{Z}_{>0} \), we define the \emph{integer games} \( k \) and \( -k \) recursively as: 
\begin{itemize}
    \item $k = \left\{ k-1 \mid \varnothing \right\}$;
    \item $-k = \left\{ \varnothing \mid -k+1 \right\}$,
\end{itemize}
where $0 = \{\varnothing \mid \varnothing \}$. 

For all odd $k\in \Z$ and $n\in \Nat$, we define the {\em dyadic rational games} recursively as: $$ \frac{k}{2^n}=\left\{ \left[\frac{k-1}{2^n}\right] \mid \left[\frac{k+1}{2^n}\right] \right\},$$ 
where the brackets denote the reduction of the fraction such that the numerator is odd, unless $0$. For example, with $k=1$ and $n=3$, the game $1/8=\{0\mid 1/4\}$.
 
By {\cite{S2013}*{Proposition 3.5}}, integer and dyadic rational games follow the standard arithmetic properties. 
For instance, the disjunctive sum of the games $1$ and $\frac{1}{2}$ equals the game $1+\frac{1}{2}=\frac{3}{2}$. 

The games $0$ and $1$ are vacuously Numbers and thus, all integers are also Numbers. Similarly, $\frac{1}{2} = \left\{0 \mid 1\right\}$ is a Number, since $0<1$. The game $\{*\mid *\}$ is also a Number since its canonical form is $0$, which is a Number. 

Henceforth, we shall call all integers and dyadic rationals simply as {\em dyadics}. Let the set of dyadics is denoted by $\mathbb{D} = \Set{ \frac{k}{2^n} \SetSymbol k\in \Z,\; n\in \Nat_0}$. Since our definition of ``Number'' is not the same as in most textbooks, we include a proof of the consistency of terminology.
\begin{theorem}\label{thm: dyadic bijection with Numbers}
    The set of dyadics 
    has a bijective relation with the set of canonical form Numbers.
\end{theorem}
\begin{proof}
    All integers are vacuously Numbers, and they are in canonical form. Next, we show that every non-integer dyadic $\frac{k}{2^n}$ is a canonical form Number. First note that $\frac{k-1}{2^n}<\frac{k+1}{2^n}$, and so it is a Number. Moreover, this game cannot be reduced (indeed, domination does not apply, and it does not reverse out). 

    Now, we will show that any literal form Number $x$ equals a dyadic. Let $y$ be the canonical form of $x$.  
    By induction, every $y^L$ and $y^R$ is dyadic. Now, by domination, $y$ will have only one left and one right option, (say) $y^L$ and $y^R$. By the simplicity theorem, $y$ is the simplest dyadic between $y^L$ and $y^R$.   
\end{proof}

\par A Number game can also be interpreted as the number of `free moves' available for Left (if positive) or Right (if negative). This raises the question of the maximum number of free moves a player is guaranteed in alternating play in a game. In CGT, the maximum number of free moves Left is guaranteed in alternating play, when starting a game, is called the Left stop of the game (this may be negative, if so, it is the minimum guaranteed loss for Left). Similarly, the negative of the maximum number of free moves Right is guaranteed when starting the game is called the Right stop of the game (this may be positive). 


\begin{definition}[Stops]\label{def: Left and Right stops}
For a game $G$, the Left stop $\Ls(G)$ and the Right stop $\Rs(G)$ are defined as:
\begin{align*}
    \Ls(G)&=\begin{cases}
          x & \text{ if } G \text{ equals a dyadic } x;\\
          \max_{G^L} \left(\Rs(G^L)\right) & \text{ otherwise;}
    \end{cases}\\
    \Rs(G)&=\begin{cases}
          x & \text{ if } G \text{ equals a dyadic } x;\\
          \min_{G^R} \left(\Ls(G^R)\right) & \text{ otherwise.}
    \end{cases}   
\end{align*}
\end{definition}
The stops of a game $G$ is the ordered pair $s(G)=(\Ls(G),\Rs(G))$. 

\begin{remark}\label{obs: characterization of non-number game}
    If the Left and Right stops of a game $G$ are not the same, then $G$ does not equal a dyadic. 
\end{remark}

Now we see a few results about the stops from \cite{S2013}.

\begin{proposition}[{\cite{S2013}*{Proposition~3.17}}]\label{prop: relation between stop and game}
    Let $G$ be a game and let $x$ be a number. Then,
    \begin{enumerate}
    \item\label{item:prop: relation between stop and game:1} $\Ls(-G) = -\Rs(G)$ and $\Rs(-G) = -\Ls(G)$;
    \item\label{item:prop: relation between stop and game:2} if $G\geq x$, then $\Ls(G)\geq \Rs(G)\geq x$. Likewise if $G\leq x$ then, $\Rs(G)\leq \Ls(G)\leq x$.
    \end{enumerate}
\end{proposition}

\begin{proposition}[{\cite{S2013}*{Proposition~3.18}}]\label{prop: Lstop>Rstop}
    Let $G$ be a short game and let $x$ be any dyadic. Then, \begin{enumerate}
        \item $\Ls(G)\geq \Rs(G)$;
        \item $\Ls(G+x)=\Ls(G)+x$ and $\Rs(G+x)=\Rs(G)+x$.
    \end{enumerate}
\end{proposition}


When we analyze games in terms of the stops, we momentarily stop thinking about winning, while rather emphasizing the stops. Sometimes we abuse language and instead of ``stops'' say {\em scores} (the loot of war). This terminology would be consistent with Milnor's positional games \cite{M1953}, where his `scoring  functions' correspond to normal play stops.\footnote{It turns out that the normal play reduced canonical forms \cite{GS2009} correspond to Milnor's positional games \cite{M1953}.}

In a disjunctive sum of games, if the first player can guarantee a higher total score by playing on a particular component, in comparison to the other components, then that component is considered `urgent'. For example, in the sum $\{4\mid -5\}+\{1\mid -2\}$ the game $\{4\mid -5\}$ is urgent compared to $\{1\mid -2\}$. 

To numerically estimate this notion of urgency, we recursively apply equal penalties to both players. The minimum penalty at which the Left and Right stops of a game become equal (i.e., the game is no longer urgent) provides an estimate of the urgency. Note that this estimate does not depend on any other components. Let $\mathbb{D}^+$ denote the set of non-negative dyadics.



\begin{definition}[Penalized Position]\label{def: penalized game}
    Let $G$ be a short game in canonical form and let $p\in \mathbb{D}^+$. Then, {\it $G$ penalized by $p$}, denoted by $G_p$, is recursively defined as 
    \begin{itemize}
        \item $G_p=\left\{{G^{\mathcal L}}_{\!p}-p\;\mid\;{G^{\mathcal R}}_{\!p}+p\right\}$ for all $0\leq p\leq t$ where $t$ is the minimum $p$ for which the Left and Right stops of $G_p$ are equal to a dyadic, say $x$,
        \item $G_p = x$ for all $p > t$.
    \end{itemize} 
    Here ${G^{\mathcal L}}_{\!p}$ denotes the set of games of the form ${G^L}_{\!p}$, and similarly for Right.
\end{definition}

\begin{example}\label{ex:temp1}
    Let $G= \{9\mid 7\}$. Then, $G_p=\{9-p\mid 7+p\}$, for all $p\leq 1$. At $p=1$, $G_p$ becomes $\{8\,\mid \,8\}$, and hence, for all $p>1$, $G_p = 8$. In summary,
     $$G_p = \begin{cases}
         \{9-p\mid7+p\} & \text{ if } p\leq 1;\\
         8 & \text{ if } 1< p.\\
     \end{cases}$$
\end{example}

\begin{observation}\label{obs: cooling and penalizing}
Although the concept of ``cooling by $t$'' defined in \cite{S2013} is similar but not identical to ``penalized by $p$'', the stops of ``$G$ penalized by $p$'' and ``$G$ cooled by $t$'' remain the same for all $p=t\geq0$.
\end{observation}

Using Observation~\ref{obs: cooling and penalizing}, we can use the results on the stops of \emph{`$G$ cooled by $t$'} for that of \emph{`$G_p$ penalized by $p$'}.

\begin{proposition}[{\cite{S2013}*{Theorem 5.11(b)}}]\label{prop: L(G_t)<L(G)}
    For any game $G$, for all $p\in \mathbb{D}^+ $, $\Ls(G)\geq \Ls(G_p)\geq \Rs(G_p)\geq \Rs(G)$. 
\end{proposition}
The middle inequality follows by Proposition~\ref{prop: Lstop>Rstop} and the main idea behind the proof of the other two inequalities is that a penalty reduces the benefit for both players. 


The minimum penalty, for which the penalized game remains no longer urgent, is the {\em temperature} of the original game.

\begin{definition}[Temperature]\label{def:temp}
    The temperature of a dyadic $G=k/2^n$ is $\te(G)=-1/2^n$, where $k\in \Z$ and $n\in\Nat \cup \Set{0}$ and if $n>0$, $k$ is an odd integer. The temperature $\te(G)$ of a non-dyadic $G$ is the smallest $p\in \mathbb{D}^+$ such that $\Ls(G_p) = \Rs(G_p)$. 
\end{definition}
If the game $G$ is given, we may write $t=\te(G)$ and $m=m(G)$, and similar for $\Ls=\Ls(G)$ and $\Rs=\Rs(G)$. 
In Example~\ref{ex:temp1}, $\te(G)=1$.  

One of the issues with the definition of temperature is that it is somewhat unwieldy from a computational  point of view. A more intuitive and appealing way of understanding (and computing) the temperature of a game comes from a more pictorial device, the {\em thermograph}. 

\begin{definition}[Thermograph]\label{def:thermograph}
    Let $G$ be a short game. Then, the thermograph of $G$, $\T(G)$, is a plot of the Left and Right stops of $G_p$ (on the $X$-axis) with respect to $p$ (on the $Y$-axis). 
\end{definition}

The minimum $p\geq0$ at which the thermograph's mast starts equals the temperature of a game. Thus, the thermograph gives us a computational means to find the temperature of a game. 

For a given (hot) game $G$, $\Ls(G_p)$ and $\Rs(G_p)$ are bounded functions on the penalty $p$. Sometimes we think of them as the walls of $\T(G)$. 

\begin{definition}[Walls]\label{def:walls}
The sets 
\begin{align*}
    \LW(G) &= \Set{(\Ls(G_p), p) \SetSymbol p \in \Dp} \text{ and} \\
    \RW(G) &= \Set{(\Rs(G_p), p) \SetSymbol p \in \Dp}
\end{align*}
are called the large left and large right walls of $\T(G)$, respectively.
The sets 
\begin{align*}
    \lw(G) &= \Set{(\Ls(G_p), p) \SetSymbol p \in \Dp,\, p \le t} \text{ and} \\
    \rw(G) &= \Set{(\Rs(G_p), p) \SetSymbol p \in \Dp,\, p \le t}
\end{align*}
are called the small left and small right walls of $\T(G)$, respectively.
A wall is either a small or a large wall.
\end{definition}

At times, we may refer to the `walls of the game' rather than explicitly stating the `walls of the thermograph of the game.' However, in both cases, we are referring to the same concept.

The large left (right) wall can be seen as an extension of the small left (right) wall, continuing indefinitely. In Figure~\ref{fig: thermograph ex for left and right walls}, we depict a thermograph where ABCD$\infty$ and ED$\infty$ is the large left and large right wall of the thermograph, respectively, while ABCD and ED is the small left and small right wall, respectively.

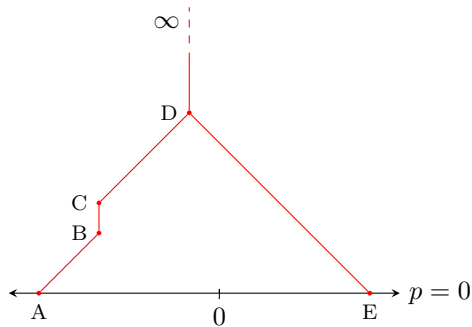
\begin{figure}[ht]
    \centering
    \begin{tikzpicture}[scale=0.4,>=stealth, dot/.style = {circle, fill=red, minimum size=#1, inner sep=0pt, outer sep = 0pt}, dot/.default = 6pt]
                    \draw [<->] (-7,0) -- (6,0) node [at end, right] {$p=0$};
                        \node[dot, scale=0.3, label=below:{\midsize A}] (A) at (-6,0) { };
                        \node[dot, scale=0.3, label=left:{\midsize B}] (B) at (-4,2) { };
                        \node[dot, scale=0.3, label=left:{\midsize C}] (C) at (-4,3) { };
                        \node[dot, scale=0.3, label=left:{\midsize D}] (D) at (-1,6) { };
                        \node[dot, scale=0.3, label=below:{\midsize E}] (E) at (5,0) { };
                        \draw[red] (A)--(B)--(C)--(D)--(E);
                        \draw[red] (D)--+(0,2) node[anchor=center] (infi) { };
                        \draw[red, dashed] (infi.north)--+(0,1.2) node[anchor=north east] {\color{black} $\infty$};
                        \draw[black] (0,0)--+(0,0.13)--+(0,-0.2) node[below] (zero) {0};
    \end{tikzpicture}
    \caption{The thermograph of $G=\Bigl\{6, \bigl\{10\mid\{5\mid3\}\bigr\}\mid-5\Bigr\}$.} 
    \label{fig: thermograph ex for left and right walls}
\end{figure}

The temperature is the $y-$coordinate of the point where the small left and right walls merge. In Figure~\ref{fig: thermograph ex for left and right walls}, the $y-$coordinate of D is the temperature.
Moreover, in \cite{S2013}*{Theorem~5.17}, there is a 
proof that the $x-$coordinate of the same point equals the mean value of the game. Thus, given a game, the thermograph gives us a computational means to find both the temperature and the mean value of the game. 
While computing the walls of a game $G$, in general, the large walls of the options need to be considered.




Sometimes, we view a thermograph as the two functions that define it, but other times, it is convenient to view it as a vertical structure. In this spirit, we define some particularly simple structures. In our proofs to come, we will assume such structures of the options by induction, and prove that they survive in the induction step.

Let $A\subset \mathbb{D}\times \mathbb{D}$ where $\mathbb{D}=\Set{k/2^n\SetSymbol k\in \Z, n\in \Nat_0}$. 
If for all $(x,y) \in A, y=kx+c$, for some constant $c$, then we say that $A$ {\em has slope}~$k$. 
 
\begin{definition}[Masts and Tents]\label{def: mast_tent}
Consider a (hot or tepid) game $G$. Then:
\begin{itemize}
    \item $G\in \M$, if $\LW(G)=\RW(G)$;
    \item $G\in \DT$ (double tent), if $\lw(G)$ has slope $-1$ and $\rw(G)$ has slope $+1$;
    \item $G\in \LST$ (left single tent), if $\lw(G)$ has slope $-1$ and $\rw(G)$ has slope $0$;
    \item $G\in\RST$ (right single tent), if $\lw(G)$ has slope $0$ and $\rw(G)$ has slope $+1$.
\end{itemize}  
\end{definition}

The terms double tent, left single tent, and right single tent refer to the shapes of a game's thermograph. These thermograph shapes can be seen in Figure~\ref{fig: tents}. 
The three categories—double tent, left single tent, and right single tent—are collectively referred to as \emph{tents}.

\begin{figure}[ht]
    \centering
    \begin{subfigure}{0.16\textwidth}
    \centering
        \begin{tikzpicture}[scale=0.5,>=stealth, dot/.style = {circle, fill=red, minimum size=#1, inner sep=0pt, outer sep = 0pt}, dot/.default = 6pt]
                    \draw [<->] (-2,0) -- (2,0) node [at end, right] { };
                        \node[dot, scale=0.3, label=below:{\midsize $\ell=r$}] (A) at (0,0) { };
                        \draw[black] (-1.3,0.13)--(-1.3,-0.13) node[anchor=north] (zero) {\midsize 0};
                        \draw[red] (A)--+(0,3);
    \end{tikzpicture}
    \caption{A Mast.} 
    \label{fig: mast}
    \end{subfigure}%
    \hfill
    \begin{subfigure}{0.28\textwidth}
    \centering
        \begin{tikzpicture}[scale=0.5,>=stealth, dot/.style = {circle, fill=red, minimum size=#1, inner sep=0pt, outer sep = 0pt}, dot/.default = 6pt]
                    \draw [<->] (-3,0) -- (3,0) node [at end, right] { };
                        \node[dot, scale=0.3, label=below:{\midsize $\ell$}] (A) at (-2,0) { };
                        \draw[black] (-0.5,0.13)--(-0.5,-0.13) node[anchor=north] (zero) {\midsize 0};
                        \node[dot, scale=0.3, label=left:{\midsize $p=t$}] (B) at (0,2) { };
                        \node[dot, scale=0.3, label=below:{\midsize$r$}] (C) at (2,0) { };
                        \draw[red] (A)--(B)--(C) (B)--+(0,1);
    \end{tikzpicture}
    \caption{A Double Tent ($\DT$).} 
    \label{fig: DT thermograph}
    \end{subfigure}%
    \hfill
    \begin{subfigure}{0.28\textwidth}
    \centering
        \begin{tikzpicture}[scale=0.5,>=stealth, dot/.style = {circle, fill=red, minimum size=#1, inner sep=0pt, outer sep = 0pt}, dot/.default = 6pt]
                    \draw [<->] (-3,0) -- (2,0) node [at end, right] { };
                        \node[dot, scale=0.3, label=below:{\midsize$\ell$}] (A) at (-2,0) { };
                         \draw[black] (1.5,0.14)--(1.5,-0.14) node[anchor=north] (zero) {\midsize 0};
                        \node[dot, scale=0.3, label=left:{\midsize$p=t$}] (B) at (0.5,2.5) { };
                        \node[dot, scale=0.3, label=below:{\midsize$r$}] (C) at (0.5,0) { };
                        \draw[red] (A)--(B)--(C) (B)--+(0,1);
    \end{tikzpicture}
    \caption{A Left Single Tent ($\LST$).}
    \label{fig: LST thermograph}
    \end{subfigure}%
    \hfill
    \begin{subfigure}{0.28\textwidth}
    \centering
        \begin{tikzpicture}[scale=0.5,>=stealth, dot/.style = {circle, fill=red, minimum size=#1, inner sep=0pt, outer sep = 0pt}, dot/.default = 6pt]
                    \draw [<->] (-3,0) -- (2,0) node [at end, right] { };
                        \node[dot, scale=0.3, label=below:{\midsize $\ell$}] (A) at (-1.5,0) { };
                        \node[dot, scale=0.3, label=left:{\midsize $p=t$}] (B) at (-1.5,2.5) { };
                        \draw[black] (-2.5,0.14)--(-2.5,-0.14) node[anchor=north] (zero) {\midsize 0};
                        \node[dot, scale=0.3, label=below:{\midsize $r$}] (C) at (1,0) { };
                        \draw[red] (A)--(B)--(C) (B)--+(0,1);
    \end{tikzpicture}
    \caption{A Right Single Tent ($\RST$).}
    \label{fig: RST thermograph}
    \end{subfigure}
    \caption{Mast and Tents. 
    }
    \label{fig: tents}
\end{figure}
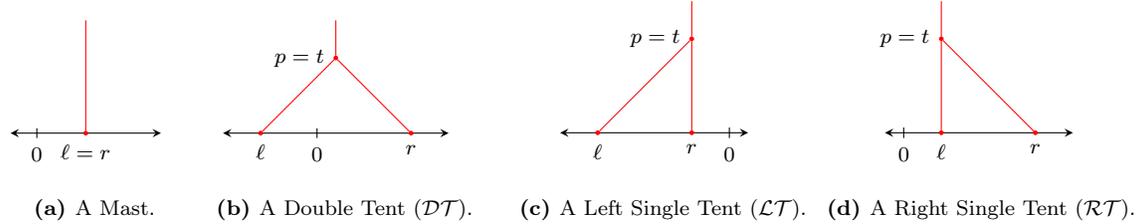
In more detail we get the following (by keeping track of the constant $c$ in the definition of slope).

\begin{observation}[Mast and Tents]\label{obs: temp of mast_tent}
 We have: 
 \begin{itemize}
     \item $G\in \DT$, if, for all $p\le t(G)$, $\Ls(G_p)= t-p+m$ and $\Rs(G_p)=p-t+m$;
     \item $G\in \LST$, if, for all $p\le t(G)$, $\Ls(G_p)= t-p+m$ and $\Rs(G_p)= m$; 
     \item $G\in \RST$, if, for all $p\le t(G)$, $\Ls(G_p)= m$ and $\Rs(G_p)= p-t+m$,
 \end{itemize}  
 where, as usual $t=t(G)$ and $m=m(G)$. 
 \end{observation}
 The following lemma helps in determining the temperatures and mean values of games with tent-shaped thermographs. 
Let us reformulate the case of $p=0$, how it applies to the proof of our main theorem. 
\begin{lemma}[Temperature and Mean of Tents]\label{lem: temp of tents}
    Let $G$ be a game. Then,  
    \begin{enumerate}
        \item if $G\in \DT$, $t=(\ell-r)/2$ and $m=(\ell+r)/2$;
        \item if $G \in \LST$, $t=\ell-r$ and $m=r$;
        \item if $G \in \RST$, $t=\ell-r$ and $m=\ell$,
    \end{enumerate}
    where, as usual $t=t(G)$, $m=m(G)$, $\ell=\ell(G)$ and $r=r(G)$.
\end{lemma}
\begin{proof}
    Apply Observation~\ref{obs: temp of mast_tent} with $p=0$.
\end{proof}

In this spirit of preparing for the main proofs to come, let us present a general lemma concerning the simplest of thermographs.

\begin{lemma}\label{lem: MvsO mast vs options}
    Consider a game $G$ and let $H$ denote the Left option of $G$ with the largest Left stop. If  $H\in\M$, then $\LW(G)$ does not depend on any other Left option than $H$.
\end{lemma}
\begin{proof}
    Recall that for a game $G$, $G_p$ denotes the game \textit{$G$ penalized by $p$}. 
    We get, for all $p\geq 0$ and all $G^L\in G^{\mathcal{L}}$,
    \begin{align}
        \Rs(H_p) = \Rs(H) &= \Ls(H) \label{eq: MvsO 1}\\
                &\geq \Ls(G^L) \label{eq: MvsO 2}\\
                &\geq \Ls({G^L}_{\!p}) \tag{By Prop~\ref{prop: L(G_t)<L(G)}}\\
                &\geq \Rs({G^L}_{\!p}) \tag{By Prop~\ref{prop: Lstop>Rstop}}
    \end{align} 
    Equation~\eqref{eq: MvsO 1} follows by Definition~\ref{def:thermograph} as $H\in \M$. Equation~\eqref{eq: MvsO 2} follows as $H$ is the option with the largest Left stop. By combining this result with the definition of Left stop, 
    we get $\Rs(G_p) = \Rs(H_p)-p$ for all $p\geq 0$.
\end{proof}

A similar result holds for the Right options. 
\begin{corollary}\label{cor: MvsO mast vs options}
    Consider a game $G$ and let $H$ denote the Right option of $G$ with the smallest Right stop. If  $H\in \M$, then  $\RW(G)$ does not depend on any other Right option than $H$.
\end{corollary}
\begin{proof}
    The proof is the same as that of Lemma~\ref{lem: MvsO mast vs options}. 
\end{proof}    
Combinatorial games are divided into 3 categories depending on the temperature:

\begin{itemize}
    \item \textbf{Hot}: A game $G$ is hot, if $t(G)>0$;
    \item \textbf{Tepid}: A game $G$ is tepid, if $t(G)=0$;
    \item \textbf{Cold}: A game $G$ is cold, if $t(G)<0$.
\end{itemize}



\section{Sherwood Organization}\label{sec: organized RH}

Let us formalize the ruleset {\sc Robin Hood}. We denote by $[a]$ the set $\Set{1,2,\dots,a}$.
     

\begin{definition}[Robin Hood]
     Let $n,a,b\in \Nat_0$. A single-heap {\sc Robin Hood} game $(n;a,b)$, where $n$ is the heap size, $a$ and $b$ are the wealths of Left and Right players, respectively, has the following options:
\begin{enumerate}
    \item the Left options are $(n-i;a,b-i)$ where $i\in\left[\min\Set{n,a}\right]$; 
    \item the Right options are $(n-j;a-j,b)$ where $j\in \left[\min\Set{n,b}\right]$. 
\end{enumerate}
\end{definition}

By convention, non-positive wealth is deemed to be 0 because a player with 0 or negative wealth cannot make a move. 
Recall that this game can also be played on multiple heaps. A {\sc Robin Hood} game on multiple heaps is same as the disjunctive sum of single heap {\sc Robin Hood} games. 

To understand a multiple heap game, it suffices to know the game values of single heap games. For this reason, from now onward in this paper, we will only consider the single heap games. However, the game values (a.k.a. canonical forms) often quickly become intractable. Let us give some intuition ``why?''.

The canonical form of the game $(4;\,2,2)$ is $G=\pm (2,\{2\mid\pm 1\})$. Suppose that we play the sum $(4;\,2,2)+(3;\,1,2)=G+\{\pm 1\mid -2\}$. Left loses if she plays to $2+\{\pm 1\mid -2\}$, but she wins if she plays to $\{\pm 1\mid -2\}+\{2\mid \pm 1\}$. In a sense, the most likely `best' move can fail depending on the surrounding context. The standard abstract way to explain this type of situation is that, indeed the game $2$ is incomparable with the game $\{2\mid \pm 1\}$, and neither option reverses out (which has to be checked). Similar arguments show that generic games of the form $(n;\, b,b)$ have $b$ canonical options for each player (all options are sensible depending on situation). Thus, the complexity of canonical form games quickly becomes intractable. However, there are some very obvious options that never come into play. 
\begin{proposition}\label{prop: dominated options}
    Consider the {\sc Robin Hood} game $(n;a,b)$, with $n>b$. If $a > b$, then  the Left options $(n-i;a,0)$, $b<i\leq \min\Set{n,a}$, are dominated by the Left option $(n-b; a,0)$. 

\end{proposition}
\begin{proof}
For $i>b$, $(n-i;a,0)=n-i<n-b=(n-b;a,0)$. 
\end{proof}
    
Thus, from now onward, we only consider the non-dominated options (with $i\le \min\Set{a,b}$). 
     
\begin{proposition}\label{prop: -heap<=stop,RH<=heap}
    Consider $n,a,b\in \Nat_0$ and let $G=(n;a,b)$. Then\begin{enumerate}
        \item $-n\leq G\leq n$;\label{prop: -heap<=stop,RH<=heap (1)}
        \item $-n\leq \Ls(G)\leq n$ and $-n\leq \Rs(G)\leq n$;\label{prop: -heap<=stop,RH<=heap (2)}
        \item $-G=(n;b,a)$.\label{prop: -heap<=stop,RH<=heap (3)}
    \end{enumerate}
\end{proposition}
\begin{proof} 
    Left can win $n -(n;\,a,b)$ and $(n;\,a,b)+n$ playing second, by playing on the number in every turn. The 2\textsuperscript{nd} item follows using the 1\textsuperscript{st} item along with Proposition~\ref{prop: relation between stop and game}. At last, the negative of a game is swapping the rules of the players.
\end{proof}

\begin{theorem}[Robin Hood Positions]\label{thm: RH positions}
    Let $G=(n;a,b)$ be a {\sc Robin Hood} position, with $a\ge b$. We have the following facts:
    \begin{enumerate}
        \item\label{item:thm: RH positions:1} $G=0$ if $a=b=0$;
        \item\label{item:thm: RH positions:2} $G=n$ if $a>0$ and $b=0$ and $G=-n$ if $a=0$ and $b>0$;
        \item\label{item:thm: RH positions:3} $G=*n$ if $n\leq b$;
        \item\label{item:thm: RH positions:4} $G$ is hot, otherwise;
        \item\label{item:thm: RH positions:5} $\Ls(G)=\max_{G^L} \Rs(G^L)$ and $\Rs(G)=\min_{G^R} \Ls(G^R)$, if $n>b$.
    \end{enumerate}
\end{theorem}
\begin{proof}
Starting with the first item, neither player has moves, and the result is trivial.

Regarding the second item, it suffices to check that $(n;a,0) - n$ is a $\mathcal P$-position. If Left starts and moves to $(n-j; a, 0)- n$, Right wins by responding with $(n-j; a, 0) -(n-1)$. This happens because, by induction, $(n - j; a, 0) = (n - j) \leq (n - 1)$. Similarly, if Right starts and moves to $(n; a, 0)-(n-1)$, Left wins by responding with $(n-1; a, 0)-(n-1)$.

In the third item, the Left options are $(n-j;a,b-j)$ where $0<j\leq n$. Since $n-j\leq \min\Set{a,b-j}$, by induction, the Left options are $0, *, *2,\dots,*(n-1)$. By the same argument, the Right options are also the same. Hence, the game $G$ is $\{0,*,\dots,*(n-1)\;|\;0,*,\dots,*(n-1)\}$, which is $*n$. 

For the fourth item, suppose without loss of generality that $a\ge b$. Then $n> b$. Thus, $\Ls(G)\ge \Rs(n-b; a,0)=n-b>0$.

    On the other hand $\Rs(G)\le \Ls(n-b;a-b,b)$ . If $a=b$, then $\Ls(n-b;a-b,b)=-(n-b)<0$, and otherwise, since the heap size $n-b$ will decrease further in computing $\max\{\Rs(n-b-i; a-b,b-i)\}$, $\Rs(G)<n-b$.
    
    
    Therefore, $\Rs(G)<\Ls(n-b;a-b,b)<n-b<\Ls(G)$ and consequently, $G$ is hot. 


    The last item is a consequence of item 4.
\end{proof}


As promised in the introduction, let us prove that {\sc Robin Hood} does not belong to hotstrat.
\begin{proof}[Proof of Theorem~\ref{thm:nohotstrat}]
Let $G_1=(11;\,1,1)$ and $G_2=(12;\,2,1)$ be two {\sc Robin Hood} games. Thus, $G_1 = \{10\mid -10\}$ and $ G_2= \{11\mid\{10 \mid -10\} \}$. Hence $t(G_1)=10$  and $t(G_2)=1$. 

    
    Let $G= G_1+G_2$. If Right starts the game $G$ by playing in the hottest component, $G_1$, to $(10;\,0,1) +G_2$, then Left can respond by playing in $G_2$ to $(10;\,0,1)+(11;\,2,0)$. This game is Left winning (by one land piece). Whereas, if Right starts the game by playing in the cooler component, $G_2$, to the game $(11;\,1,1)+(11;\,1,1)$, Right wins by henceforth mimicking Left's moves.  
\end{proof}


\section{Robin Hood meets Little John}\label{sec:RH meets Little John}

\renewcommand{\arraystretch}{1.3}%
         


We know that the \textsc{Robin Hood} positions with sufficiently large heap sizes are hot. By this, we mean that both players benefit from starting the game. Is the advantage the same regardless of the initial move each player makes, or is there a specific move that offers the greatest benefit? Since reduced wealth is disadvantageous for a player, the move that maximizes the reduction of the opponent's wealth is likely to yield the highest advantage.

\begin{definition}[Little John Move]
    Consider a {\sc Robin Hood} game $(n;\, a,b)$, where $n>0$ and  $\max\Set{a,b}>0$. If $b>0$, Left's Little John option is $(n-\min\Set{n, a, b};\; a,\;b-\min\Set{n, a, b})$, and otherwise  it is $(n-1;\, a,0)$. Right's Little John option is analogously defined.
\end{definition}

\begin{definition}[Little John Path]
    A Little John path is a sequence of alternating play Little John moves. 
\end{definition}
After a Little John path, there will usually be a number of free Little John moves for either player, depending on the size of $n$. 
    In Figure~\ref{fig: Little John Path}, we illustrate the Little John path on $(n;a,b)$ if $a/b = 1.4$.
    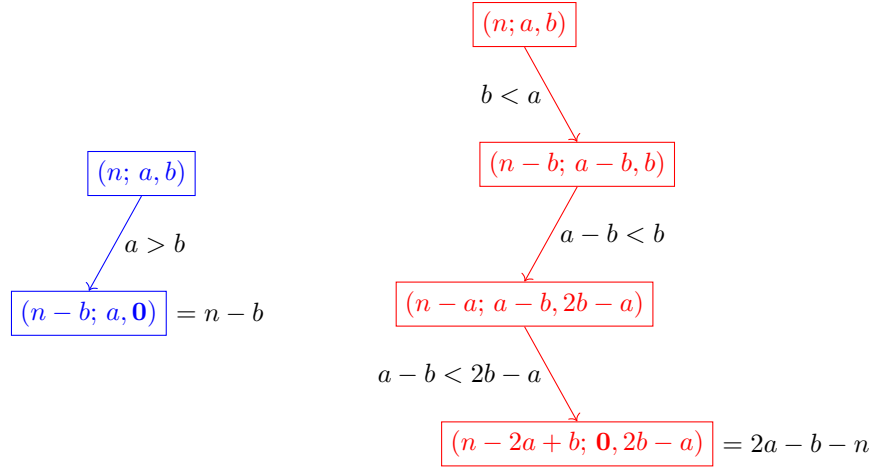
\begin{figure}[ht]
        \centering
        \begin{subfigure}[b]{0.40\textwidth}
        \centering
            \begin{tikzpicture}[anode/.style={draw=blue}, scale=0.7]
                 \node[anode] (a) at (0,0) {\color{blue} $(n;\,a,b)$};
                 \draw[->,blue] (a.south) -- +(-1,-1.8) node[anode,anchor=north] (b) {$(n-b;\,a,\bm 0)$} node [midway,  right] (ab1) {\color{black} $a>b$} node [midway, sloped, above] (ab2) { };
                 \node[anchor=west] (e) at (b.east) {$=n-b$};
                 \node (c) at (-1,-5.4) { };
            \end{tikzpicture}
            \caption{The Little John Path when Left starts.}
        \end{subfigure}%
        \hspace{0em}%
        \begin{subfigure}[b]{0.40\textwidth}
        \centering
            \begin{tikzpicture}[anode/.style={draw=red},scale=0.7]
                 \node[anode] (a) at (0,0) {\color{red}$(n;a,b)$};
                 \draw[->,red] (a.south) -- +(1,-1.8) node[anode,anchor=north] (b) {$(n-b;\,a-b,b)$} node [midway, left] (ab1) {\color{black} $b<a$} node [midway, above] (ab2) { };
                 \draw[->,red] (b.south)--+(-1,-1.8) node[anode,anchor=north] (c) {$(n-a;\,a-b,2b-a)$} node [midway, above] (bc2) {} node [midway, right] (ab1) {\color{black} $a-b<b$};
                 \draw[->,red] (c.south)--+(+1,-1.8) node[anode,anchor=north] (d) {$(n-2a+b;\,\bm 0,2b-a)$} node [midway, above] (cd2) {} node [midway, left] (ab1) {\color{black} $a-b<2b-a$};
                 \node[anchor=west] (e) at (d.east) {$=2a-b-n$};
            \end{tikzpicture}
            \caption{The Little John Path when Right starts.}
        \end{subfigure}
        \caption{The Little John Paths of $(n;a,b)$, for $a/b = 1.4$ and sufficiently large $n$.}
        \label{fig: Little John Path}
    \end{figure}
If Left and Right have wealths of $27$ and $17$ respectively, the sequence of pairs (Left's wealth, Right's wealth) along the Little John path when Right starts the game is $(27,17)$, $(10,17)$, $(10,7)$, $(3,7)$, $(3,4)$, $(0,4)$. Observe that the wealths of Left and Right changes alternatively until one of the players' wealth reaches 0. Thus, we can simplify and reduce this sequence to $27,\; 17,\; 10,\; 7,\; 3,\; 4,\; 0$, where the first value belongs to Left, the second to Right, the third to Left, and so on. If we omit the $0^{\text{th}}$ term from the reversed sequence of players' wealth on the Little John path, the resulting sequence ($4,\;3,\;7,\;10,\;17,\;27$) follows the pattern where each term is the sum of the two preceding terms. The next section explores the properties of such sequences.


\section{Pingala sequences and their properties}\label{sec:ping}

The previous section points at classical number sequences. 
\begin{definition}[Pingala Sequence]
 The Pingala Sequence $(P_k)_{k\geq 0}$ is given by $P_0=0$, $P_1=1$ and, for all  $k\geq 0$,  $P_{k+2} = P_{k+1}\;+\;P_k $.\footnote{This is commonly known as the Fibonacci sequence. We attribute the sequence instead to Acharya Pingala, an ancient (3rd–2nd century BCE) Indian mathematician and poet who used this sequence in poetry long before Fibonacci lived.}
\end{definition}

The wealths of the players may not lie in this sequence so we define a modified version of it. 

\begin{definition}[Modified Pingala Sequence]\label{def: MPS}
    A sequence $(U_k)_{k\geq 0}$ is a Modified Pingala Sequence (MP-sequence) if $U_0$ and $U_1$ are positive integers such that $U_0\geq U_1$, and $U_{k+2} = U_{k+1}+U_{k}$ for all $ k\geq 0$. 
\end{definition}

\begin{remark}
    The Pingala sequence includes 0 as its $0^{\text{th}}$ term. However, as noted earlier, the reversed sequence of players' wealth follows the Pingala sequence pattern only when 0 is excluded. The condition $U_0,\;U_1\in \Z_{>0}$ in the definition of the Modified Pingala sequence ensures that 0 and any negative numbers are omitted from the sequence. 
\end{remark}

Next, we see some sequences derived from the Pingala sequence and their properties.

\begin{definition}[Ratio Sequences]
    The ratio sequences $(O_k)_{k\geq 0}$ and $(E_k)_{k\geq 0}$ are given by, for all $ k\geq 0$, $O_k = P_{2k+2}/P_{2k+1}$ and $E_k = P_{2k+3}/P_{2k+2}$.
\end{definition}

We refer to $(O_k)_{k \geq 0}$ and $(E_k)_{k \geq 0}$ as the Odd and Even Ratio sequences, respectively.

It is well known that Even and Odd ratio sequences are strictly decreasing and increasing, respectively, and both converge to the golden ratio.

The following proposition is a routine.
\begin{proposition}\label{prop: unique U_n exists}
    Consider two positive integers $a$ and $b$ such that $a\geq b$. Then, 
    \begin{enumerate}
        \item there exists a unique MP-sequence $(U_k)_{k\geq 0 }$, such that $U_{\mu}=b$ and $U_{\mu+1}=b$, for some $\mu\ge 0$; 
        \item there exists a unique MP-sequence $(V_k)_{k\geq 0 }$, such that $V_{\nu}=a$ and $V_{\nu+1}=b$, for some $\nu\ge 0$.
    \end{enumerate}
\end{proposition}
\begin{proof}
    The proof follows simply by generating the sequence by using the MP-sequence rules that every term is sum of the last two terms and the starting term in the sequence is greater than equal to the next term.
\end{proof}
As an illustrative example, consider $a=31$ and $b=20$. Then $(U_k)=(7,2,9,11,20,31,51,82,\dots )$, with  $\mu=4$, and $(V_k)= (31,20,51,71,122,\dots )$, with $\nu=0$. 

Let us list a few elementary properties of MP-sequences.
\begin{observation}\label{obs: increasing MPS}
   If $(U_k)_{k\geq 0}$ is a MP-sequence, then the following hold:
   \begin{enumerate}
       \item\label{item:obs: increasing MPS:1} $U_1\leq U_0 < U_2 < U_3<\dots<U_n<\dots$; 
       \item\label{item:obs: increasing MPS:2} $U_k \in \mathbb Z_{>0}$ for all $k\geq 0 $.
   \end{enumerate}
\end{observation}
The next proposition determines the $\mu$ and $\nu$ from Proposition~\ref{prop: unique U_n exists}, given $a$ and $b$. First, we relate the ratio of consecutive terms of a MP-sequence with those of the Pingala sequence. 

\begin{lemma}\label{prop: PS-MPS inequality}
    Let $(U_k)_{k\geq 0}$ be an MP-sequence. Then, for all $k\ge 1$, for all $i$ such that $0\leq 2i \leq k-1 $, 
\begin{enumerate}
     \item $U_{k-2i}/U_{k-2i-1}\leq  O_j \Leftrightarrow U_{k+1}/U_k\geq E_{j+i}$, and \label{item:prop: PS-MPS inequality:1}
     \item $U_{k-2i+1}/U_{k-2i}\leq O_j \Leftrightarrow U_{k+1}/U_k\leq O_{j+i}$.\label{item:prop: PS-MPS inequality:2}
\end{enumerate}
\end{lemma}
\begin{proof}
    Let $k$ be fixed. We will prove both statements together by induction on $i$.\\
    \emph{Base Case:} Let $i=0$. 
        \begin{enumerate}[wide, labelindent=0pt, parsep=7pt, topsep=5pt]
            \item Then the following statements are equivalent. 
                \begin{align}
                    \frac{U_{k}}{U_{k-1}} & \leq O_j \left(=  \frac{P_{2j+2}}{P_{2j+1}}\right)\label{eq:6.6.1}\\
                    \frac{U_{k+1}-U_k}{U_k} & \geq  \frac{P_{2j+1}}{P_{2j+2}} \label{eq:6.6.2}\\
                    \frac{U_{k+1}}{U_k} & \geq \frac{P_{2j+3}}{P_{2j+2}}\: (=E_{j+0}).\label{eq:6.6.3}
                \end{align}
                Equation~\eqref{eq:6.6.2} follows by inverting the Equation~\eqref{eq:6.6.1}.
            \item For $i=0$, both sides are the same.
        \end{enumerate} 
    Suppose that the statement is true for $i$ and let $2(i+1)\leq k-1$. Then, for part (1), the following statements are equivalent by induction on $i$.
             
            \begin{align}
                    \frac{U_{k-2(i+1)}}{U_{k-2(i+1)-1}}&\leq \frac{P_{2j+2}}{P_{2j+1}} \label{eq:6.6.4}\\
                     \frac{U_{k-2i-1}-U_{k-2i-2}}{U_{k-2i-2}}&\geq \frac{P_{2j+1}}{P_{2j+2}}\label{eq:6.6.5}\\
                     \frac{U_{k-2i-1}}{U_{k-2i-2}}&\geq \frac{P_{2j+3}}{P_{2j+2}}\label{eq:6.6.6}\\
                     \frac{U_{k-2i}}{U_{k-2i-1}}&\leq \frac{P_{2(j+1)+2}}{P_{2(j+1)+1}} =O_{j+1} \label{eq:6.6.7} \\
                     \frac{U_{k+1}}{U_k} &\geq E_{j+1+i} \tag{by induction}. 
                \end{align}
            Equation~\eqref{eq:6.6.5} follows by inverting Equation~\eqref{eq:6.6.4} and Equation~\eqref{eq:6.6.7} follows by inverting Equation~\eqref{eq:6.6.6} and adding 1.
            This completes the induction.
            We skip the proof of part (2) as it is similar to part (1).
\end{proof}



\begin{proposition}\label{lem: position of b}
    Let $a,b>0$ be integers and let $(U_k)_{k\geq 0}$ be the unique MP-sequence such that $U_{\mu}=b,\; U_{\mu+1}=a$ for some $\mu\geq 0$. Then $\mu$ is determined as follows: 
    \begin{enumerate}
        \item If $\frac{a}{b}< \phi$, then $\mu=2\widehat\mu$, where $\widehat\mu=\min\{k\geq 0: \frac{a}{b}\leq O_k\}$;
        \item If $\frac{a}{b}> \phi$, then $\mu=2\widetilde\mu+1$, where $\widetilde\mu =\min\{k\geq 0: \frac{a}{b}\geq E_k\}$.
    \end{enumerate} 
\end{proposition}
\begin{proof}
 The indexes $\widehat\mu$ and $\widetilde\mu$ always exist because the sequences $(O_k)_{k\geq0}$ and $(E_k)_{k\geq0}$ are increasing and decreasing, respectively and both converge to $\phi$.
    
    \begin{enumerate}[wide, labelindent=0pt, parsep=7pt, topsep=5pt]
        \item The following inequalities are equivalent:
            \begin{align}
                \frac{a}{b} &\leq O_{\widehat\mu} \label{eq:lem: position of b:1}\\
                \frac{U_{\mu+1}}{U_\mu} &\leq  O_{\widehat\mu} \nonumber\\
                 \frac{U_{\mu-2{\widehat\mu}+1}}{U_{\mu-2{\widehat\mu}}} &\leq O_0 \left(=1\right) \label{eq:lem: position of b:2}\\
                 U_{\mu-2{\widehat\mu}+1} &\leq U_{\mu-2{\widehat\mu}}\\
                 \mu-2{\widehat\mu} &=0 & \tag{Obs~\ref{obs: increasing MPS}(\ref{item:obs: increasing MPS:1})}
            \end{align}
            Equation~\eqref{eq:lem: position of b:1} follows by the definition of ${\widehat\mu}$ and Equation~\eqref{eq:lem: position of b:2} follows by Proposition~\ref{prop: PS-MPS inequality}(\ref{item:prop: PS-MPS inequality:2}) for $i={\widehat\mu}$ and $j=0$.
        \item The following inequalities are equivalent:
            \begin{align}
                \frac{a}{b}&\geq E_{\widetilde\mu} \label{eq:lem: position of b:3}\\
                \frac{U_{\mu+1}}{U_{\mu}} &\geq E_{\widetilde\mu} \nonumber\\
                \frac{U_{\mu-2{\widetilde\mu}}}{U_{\mu-2{\widetilde\mu}-1}} &\leq O_0 (=1) \label{eq:lem: position of b:4}  \\
                U_{\mu-2{\widetilde\mu}} &\leq U_{\mu-2{\widetilde\mu}-1}\nonumber\\
                \mu-2{\widetilde\mu}-1&=0 &\tag{Obs~\ref{obs: increasing MPS}(\ref{item:obs: increasing MPS:1})}
            \end{align}  
            Equation~\eqref{eq:lem: position of b:3} follows by the definition of ${\widetilde\mu}$ and Equation~\eqref{eq:lem: position of b:4} follows by Proposition~\ref{prop: PS-MPS inequality}(\ref{item:prop: PS-MPS inequality:1}) for $i={\widetilde\mu}$ and $j=0$.
    \end{enumerate} 
    This completes the proofs of both statements.
\end{proof}

We now introduce the Pingala sequence with alternating signs, which will be used to establish a relation between two MP-sequences and hence two games.

\begin{definition}\label{def: APS}
    The alternating pingala sequence $\left(\bar{P}_k\right)_{k\geq 0}$ is given by $\bar{P}_k=(-1)^{k+1}P_k$.
\end{definition}
\begin{proposition}\label{Prop:sum of pingala}
     The following equalities hold.
     \begin{enumerate}
         \item $\bar{P}_{k+2} = \bar{P}_{k}-\bar{P}_{k+1} \text{ for all } k\geq 0$;
         \item $\sum_{i=1}^k P_i=P_{k+2}-1 \text{ for all } k\geq 1$;
         \item $\sum_{i=1}^{k} \bar{P}_i = 1+(-1)^{k-1}P_{k-1} \text{ for all } k\geq 1$; 
         \item For any MP-sequence $(U_k)_{k\geq 0}$, $\sum_{i=1}^{k} U_i = U_{k+2}-U_1-U_0 \text{ for all } k\geq 1 $.
     \end{enumerate}
\end{proposition}
\begin{proof}
    The proofs of all the statements follow by standard induction arguments.
\end{proof}


Later we will consider two {\sc Robin Hood} games for which the wealth of one of the players differ by one, while the other player's wealth remains the same. To compare them, we here give a relation between the MP-sequences generated by their respective pair of wealths.

\begin{proposition}\label{prop: new old MPS relation}
    Let $(U_n)_{n\geq 0}$ and $(V_n)_{n\geq 0}$ be two MP-sequences. Suppose there exist constants  $\alpha \text{ and } \beta$ such that: \begin{enumerate}
        \item $V_{\alpha}=U_\beta$ and $V_{\alpha+1} = U_{\beta+1}+1$. Then, for all $  0\leq k\leq \min\Set{\alpha,\beta}$, $V_{\alpha-k} = U_{\beta-k}+\bar{P}_{k}$;
        \item $ V_{\alpha}=U_\beta+1$ and $V_{\alpha+1} = U_{\beta+1}$. Then, for all $  0\leq k\leq \min\Set{\alpha,\beta}$,  $V_{\alpha-k} = U_{\beta-k}+\bar{P}_{k+1}$.
    \end{enumerate}      
\end{proposition}
\begin{proof}
We only prove part 1, as the proof of part 2 is similar. We  induct on $k$. 
        For $k=0$ the statement follows by the assumption, and for $k=1$ (assume $\min\Set{\alpha,\beta}\geq 1$) we have  
        \begin{align*}
            V_{\alpha-1} &= V_{\alpha+1}-V_\alpha \\
                         &= U_{\beta+1}+1-U_{\beta} \\
                         &= U_{\beta-1}+\bar{P}_{1},
        \end{align*}
        where the first and last equalities follow by Definition~\ref{def: MPS}.
        
        Suppose the statement holds for all $k\leq n \text{ where } n \text{ satisfies } 1\leq n+1\leq \min\Set{\alpha,\beta}$.
        Then, by the induction hypothesis $V_{\alpha-(n-1)} = U_{\beta-(n-1)}+\bar{P}_{n-1} \text{ and } V_{\alpha-n} = U_{\beta-n}+\bar{P}_{n}$, and hence, 
        \begin{align*}
            V_{\alpha-(n+1)} &= V_{\alpha-(n-1)} - V_{\alpha-n}\\
                             &= U_{\beta-(n-1)}-U_{\beta-n} + \bar{P}_{n-1}-\bar{P}_{n}\\
                             &= U_{\beta-(n+1)} + \bar{P}_{n+1},
        \end{align*}
        where the first and last equalities follow by Definition~\ref{def: MPS}. 
\end{proof}

The next section concerns the stops of {\sc Robin Hood}, which, we will see, are the same as the stops on the Little John path.

\section{Little John's Ruleset}\label{sec:LJ}


Intuitively, whenever there is something to gain at the end in alternating play, that is, whenever the game is hot, the players seek to minimize the opponent's wealth. Namely, by doing so, they reduce the opponent's wealth reducing power, and so on. So, in order to understand the stops, by following this intuition, it makes sense to focus on the Little John path. This wisdom suggests a variation of {\sc Robin Hood} with exclusively  Little John moves.  

\begin{definition}[{\sc Little John}] Let $(n;a,b)^*$ denotes a position of the ruleset {\sc Little John} on a heap of size $n$. 
\begin{enumerate}
    \item if $n=0$ or $a=0=b$, neither player has any option, and otherwise:
    \item if $a,b>0$, the only Left option is $(n-\gamma;a,b-\gamma)^*$ and the only Right option is $(n-\gamma;a-\gamma,b)^*$ where $\gamma=\min\Set{n,a,b}$;
    \item if $a>0=b$, the only Left option is $(n-1;a,0)^*$ and Right has no option. Similarly, if $b>0=a$, the only Right option is $(n-1;0,b)^*$ and Left has no option. 
\end{enumerate}    
\end{definition}


The next propositions and lemma allow us to compute the stops of {\sc Little John}. For simplicity of notations, we remove the extra set of bracket from $\Ls((n;a,b)^*)$ and write $\Ls(n;a,b)^*$ and follow similar notion for Right stop.

\begin{proposition}\label{prop: -heap<=stop,LJ<=heap}
    Consider $n,a,b\in \Nat_0$ and let $G=(n;a,b)^*$. Then \begin{enumerate}
        \item\label{item:prop: -heap<=stop,LJ<=heap:1} $-n\leq G\leq n$;
        \item\label{item:prop: -heap<=stop,LJ<=heap:2} $-n\leq \Ls(G)\leq n$ and $-n\leq \Rs(G)\leq n$.
    \end{enumerate}
\end{proposition}
\begin{proof} The proof is same as that of Proposition~\ref{prop: -heap<=stop,RH<=heap}.
\end{proof}

\begin{lemma}\label{lem: Wise RH is hot}
    Consider $n,a,b\in \Nat_0 \text{ such that } n\geq \min\Set{a,b} \text{ and } a\geq b>0$. Then
    \begin{enumerate}
        \item $(n;a,b)^*$ is hot;
        \item\label{item:lem: Wise RH is hot:2} $\Ls(n;a,b)^*=\Rs(n-b;a,0)^*$ and $\Rs(n;a,b)^*=\Ls(n-b;a-b,b)^*$.
    \end{enumerate}
\end{lemma}
\begin{proof}
The proof is similar to that of Theorem~\ref{thm: RH positions}.
\end{proof}

The Left and Right stops of a game $G$ are the same as the game value if and only if $G$ is a Number. Since Numbers are cold, and Lemma~\ref{lem: Wise RH is hot} establishes that {\sc Little John} positions with large $n$ and positive wealths are hot, these positions cannot be Numbers. Consequently, the Left and Right stops of such {\sc Little John} positions differ from their game value. The next lemma compute these stops.

\begin{lemma}[Little John Stops]\label{lem: stop value}
     Consider $n,a,b \in \Nat_0$ and suppose $n\geq a+b$. 
    \begin{enumerate}
        \item\label{item:lem: stop value:1} If $a=b=0$, then  $\Rs(n;a,b)^*=0$.
        \item\label{item:lem: stop value:2} If $b>0=a$, then $\Rs(n;a,b)^*=-n$ and if $a>0=b$, then  $\Rs(n;a,b)^*=n$.
        \item\label{item:lem: stop value:3} If $b\geq a>0$, then $\Rs(n;a,b)^*=-(n-a)$.
        \item If $a>b>0$, let $(U_i)_{i\geq 0}$ denote the unique MP-sequence such that $U_\mu=b$ and $U_{\mu+1}=a$. 
        \begin{enumerate}
            \item If\label{item:lem: stop value:5} $\frac{a}{b}<\phi$, then 
            $\Rs(n;a,b)^*=a+b-n-U_0$.
            \item\label{item:lem: stop value:6} If $\frac{a}{b}>\phi$, then  $\Rs(n;a,b)^*=n-(a+b)+U_0$.
        \end{enumerate}
    \end{enumerate}   
\end{lemma}
\begin{proof}
    To calculate $\Rs(n;a,b)^*$, we observe the following: 
\begin{enumerate}[wide, labelindent=0pt, parsep=10pt, topsep=5pt]
    \item For all $n\geq 0$, $\Rs(n;0,0)^*=0$ as $(n;0,0)^*=0$.
    \item If $b>0=a$, then by Theorem~\ref{thm: RH positions}(\ref{item:thm: RH positions:2}), 
    $\Rs(n;0,b)^*=\Rs(-n)=-n$ for all $n\geq 0$. Similarly, if $a>0=b$, then we have $\Rs(n;a,0)^*=\Rs(n)=n$ for all $n\geq 0$ using Theorem~\ref{thm: RH positions}(\ref{item:thm: RH positions:2}).
    
    \item if $0<\frac{a}{b}\leq 1$, then $\Rs(n;a,b)^*=\Ls(n-a;0,b)^*=-(n-a)$. The first equality holds for all $n\geq a+b$, by Lemma~\ref{lem: Wise RH is hot}. The second equality holds as $(n-a;0,b)^*=-(n-a)$ by Theorem~\ref{thm: RH positions}(\ref{item:thm: RH positions:2}).
    
    \item \begin{enumerate}
        \item Case $1<\frac{a}{b}<\phi$.  Recall Lemma~\ref{lem: position of b} which says $\mu=2\widehat\mu $ where $\widehat\mu=\min\Set{k\geq 0\SetSymbol\frac{a}{b}\leq O_k}$. Then, for all $n\geq a+b$,
            \begin{align}
                \Rs(n;a,b)^* &=\Rs(n;\;U_{2\widehat\mu+1},U_{2\widehat\mu})^* \nonumber\\
                &=\Ls(n-U_{2\widehat\mu};\;U_{2\widehat\mu-1},U_{2\widehat\mu})^* \label{eq:4.1}\\
                &=\Rs(n-\left(\sum_{i=2\widehat\mu-1}^{2\widehat\mu}U_i\right);\;U_{2\widehat\mu-1},U_{2\widehat\mu-2})^* \quad  \dots \label{eq:4.2}\\
                &=\Rs(n-\left(\sum_{i=1}^{2\widehat\mu}U_i\right);\;U_1,U_0)^* \label{eq:4.3}\\
                &=\Ls(n-\left(\sum_{i=1}^{2\widehat\mu}U_i\right)-U_1;\;0,U_0)^* \label{eq:4.4}\\
                &=-(n-\left(\sum_{i=1}^{\mu}U_i\right)-U_1) \label{eq:4.5}\\
                &= -(n-U_{\mu+2}+U_0) \label{eq:4.6}\\
                &= -(n-(a+b)+U_0). \nonumber
            \end{align}
        Equation~\eqref{eq:4.1} follows using Observation~\ref{obs: increasing MPS}, Lemma~\ref{lem: Wise RH is hot}(\ref{item:lem: Wise RH is hot:2}), and the fact that, at each step, the heap size is at least the sum of the players' wealths, as $n\geq a+b$, and both decrease equally in each iteration. By repeating this process, we get Equations~\eqref{eq:4.2}, \eqref{eq:4.3} and \eqref{eq:4.4}. Equation~\eqref{eq:4.5} holds using case~(2) of this proof and Proposition~\ref{prop: relation between stop and game}. Equation~\eqref{eq:4.6} is obtained using Proposition~\ref{Prop:sum of pingala} as $\frac{a}{b}>1 \implies \mu>0$. The final equality holds as $U_{\mu+1}=a \text{ and } U_\mu=b$.
        
        \item Case $\frac{a}{b}>\phi$. From Lemma~\ref{lem: position of b}, we have $\mu=2\widetilde\mu+1$ where $\widetilde\mu=\min\Set{k\geq 0\SetSymbol\frac{a}{b}\geq E_k}$. Then, for all $n\geq a+b$,
            \begin{align}
                \Rs(n;a,b)^* &=\Rs(n;\;U_{2\widetilde\mu+2},U_{2\widetilde\mu+1})^* \nonumber \\
                & = \Ls(n-U_{2\widetilde\mu+1};\;U_{2\widetilde\mu},U_{2\widetilde\mu+1})^* \label{eq:4.7} \\
                & = \Rs(n-\left(\sum_{i=2\widetilde\mu}^{2\widetilde\mu+1}U_i\right);\;U_{2\widetilde\mu},U_{2\widetilde\mu-1})^* \quad \dots  \label{eq:4.8}\\
                & = \Ls(n-\left(\sum_{i=1}^{2\widetilde\mu+1}U_i\right);\;U_0,U_1)^* \label{eq:4.9} \\ 
                & = \Rs(n-\left(\sum_{i=1}^{2\widetilde\mu+1}U_i\right)-U_1;\;U_0,0)^* \label{eq:4.10} \\
                & = (n-\left(\sum_{i=1}^{\mu}U_i\right)-U_1) \label{eq:4.11} \\
                & = n-U_{\mu+2}+U_0 \label{eq:4.12} \\
                & = n-(a+b)+U_0. \nonumber   
            \end{align}
        Equation~\eqref{eq:4.7} follows using Observation~\ref{obs: increasing MPS} and Lemma~\ref{lem: Wise RH is hot}(\ref{item:lem: Wise RH is hot:2}) and the fact that, at each step, the heap size is at least the sum of the players' wealths, as $n\geq a+b$, and both decrease equally in each iteration. By repeating this process, we get Equations~\eqref{eq:4.8}, \eqref{eq:4.9} and \eqref{eq:4.10}. Equation~\eqref{eq:4.11} holds using case~(2) of this proof and Equation~\eqref{eq:4.12} follows using Proposition~\ref{Prop:sum of pingala}. The last equality holds as $U_{\mu+1}=a \text{ and } U_\mu=b$.
    \end{enumerate}    
\end{enumerate} The proof is now complete, as every item has been addressed.   
\end{proof}


Having more wealth does not hurt a {\sc Little John} player.  The next theorem compares players' benefit if wealth of one of the player is increased. This is stop monotonicity with respect to wealth. (Later, in Lemma~\ref{lem: osm (option stop monoto)}, we will encounter also stop monotonicity with respect to {\sc Robin Hood} option played.)


\begin{theorem}[Little John Stop Monotonicity]\label{thm: stop monotonicity}
Consider $n,a,b \in \mathbb{N}\cup \Set{0}$. Then, for all $n\geq a+b+1$,
    \begin{enumerate}
        \item $\Rs(n;a,b)^* \leq \Rs(n;a+1,b)^*$;
        \item $\Rs(n;a,b+1)^* \leq \Rs(n;a,b)^*$; 
         \item $\Ls(n;a,b)^* \leq \Ls(n;a+1,b)^*$;
        \item $\Ls(n;a,b+1)^* \leq \Ls(n;a,b)^*$.   
    \end{enumerate}
\end{theorem}
\begin{proof}
    We prove the first item and the other are similar.
    
    Let $(U_i)_{i\geq 0}$ be the unique MP-sequence such that for some $\mu\geq 0$, $U_\mu=\min\Set{a,b}$ and $U_{\mu+1}=\max\Set{a,b}$. The uniqueness follows by Proposition~\ref{prop: unique U_n exists}. 
    
    Similarly, let $(V_i)_{i\geq 0}$ be the unique MP-sequence such that for some $\nu\geq 0$, $V_\nu=\min\Set{a+1,b}$ and $V_{\nu+1}=\max\Set{a+1,b}$.
    We define $$R_U \coloneqq \Rs(n;a,b)^*,~R_V \coloneqq \Rs(n;a+1,b)^*$$
    
    We need to prove $R_U\leq R_V$. There are 5 cases based on the relative values of $a$ and $b$ with respect to the golden ratio. We will use the Stop Values Lemma~\ref{lem: stop value} several times here.
    \begin{enumerate}[wide,  labelindent=0pt, parsep=3pt, label=(\alph*)]
            \item If $a=b=0$, then, by Lemma~\ref{lem: stop value}(\ref{item:lem: stop value:1}), $R_U=0\leq n= R_V$ for all $n\geq 1$.
            \item If $a>0=b$, then, by Lemma~\ref{lem: stop value}(\ref{item:lem: stop value:2}), $R_U=n=R_V$ for all $n\geq 0$.
            \item If $b>0=a$, then, $R_U=-n\leq R_V$ for all $n\geq 0$, where the first equality holds using Lemma~\ref{lem: stop value}(\ref{item:lem: stop value:2}) and second inequality holds using Proposition~\ref{prop: -heap<=stop,LJ<=heap}.
            \item Case $\phi<\frac{a}{b}. $ 
            By Lemma~\ref{lem: stop value}, we have,
                \begin{align}
                    R_U &= n-(a+b)+U_0,\label{eq:sm 1}\\                    
                    R_V &= n-(a+1+b)+V_0,\label{eq:sm 2}
                \end{align} for all $n\geq a+b+1$. Now, to compare $R_U$ and $R_V$, we compare $U_0$ and $V_0$. Recall that $V_{\nu+1}=a+1=U_{\mu+1}+1$ and $V_\nu=b=U_\mu$, as $\frac{a}{b}>\phi$. Thus, by Proposition~\ref{prop: new old MPS relation}, we have, for all $0\leq k\leq \min\Set{\mu,\nu}$,
                \begin{equation}\label{eq:sm 3}
                    V_{\nu-k}=U_{\mu-k}+\bar{P}_k.
                \end{equation} 
                Now, by Lemma~\ref{lem: position of b}, $\mu=2\widetilde{\mu}+1$ and $\nu=2\widetilde{\nu}+1$ where
            $$\widetilde{\mu}=\mathrm{min}\left\{i\geq 0:\frac{a}{b}\geq E_i\right\} \text{ and } \widetilde{\nu}=\mathrm{min}\left\{i\geq 0:\frac{a+1}{b}\geq E_i\right\}.$$
            Recall that the sequence $\left(E_i\right)_{i\geq 0}$, where $E_i = P_{2i+3}/P_{2i+2}$, is decreasing.
                Therefore, $ \widetilde{\nu}\leq \widetilde{\mu}$ and hence, $\nu\leq \mu$. Now, 
                \begin{align}
                    R_V-R_U &= V_0-1-U_0\tag{by Eq~\eqref{eq:sm 1}-\eqref{eq:sm 2}}\\
                            &=U_{\mu-\nu}+\bar{P}_\nu-1-U_0 \tag{by Eq~\eqref{eq:sm 3}}\\
                            &=U_{2(\widetilde{\mu}-\widetilde{\nu})}+\bar{P}_{2\widetilde{\nu}+1}-1-U_0 \tag{$\nu=2\widetilde{\nu}+1,\;\mu=2\widetilde{\mu}+1$}\\
                            &\geq U_0+P_{2\widetilde{\nu}+1}-1-U_0\label{eq:sm 4}\\
                            &\geq 0 \tag{$P_{2\widetilde{\nu}+1}\geq 1$}.
                \end{align} Equation~\eqref{eq:sm 4} follows using the following facts:
                \begin{itemize}
                    \item if $\widetilde{\mu}=\widetilde{\nu}$, then $U_{2(\widetilde{\mu}-\widetilde{\nu})}=U_0$;
                    \item if $\widetilde{\mu}>\widetilde{\nu}$, then by Observation~\ref{obs: increasing MPS}, $U_{2(\widetilde{\mu}-\widetilde{\nu})}\geq U_2>U_0$;
                    \item $\bar{P}_{2\widetilde{\nu}+1}=P_{2\widetilde{\nu}+1}$ by Definition~\ref{def: APS}.
                \end{itemize}
             
                 \item Case $\frac{a}{b}<\phi<\frac{a+1}{b}$. By Lemma~\ref{lem: stop value}, we have,
                 \begin{align*}
                     R_U & = \begin{cases}
                         a-n  & \text{if } a\leq b~;\\
                         a+b-n-U_0  & \text{if } a>b~,
                     \end{cases}\\
                     R_V & = n-(a+1+b)+V_0,                 
                 \end{align*} for all $n\geq a+b+1$.
                 Hence $R_U<0<R_V$.

                 
                 \item Case  $\frac{a+1}{b}<\phi$. By Lemma~\ref{lem: stop value}, we have, 
                 \begin{align}
                    R_U &= \begin{cases}
                                    a-n  & \text{if } a\leq b~;\\
                                    a+b-n-U_0 & \text{if } a>b~, 
                            \end{cases}\label{eq:thm:stop monotonicity:1}\\
                    R_V &= \begin{cases}
                                    a+1-n  & \text{if } a+1\leq b~;\\
                                    a+1+b-n-V_0 & \text{if } a+1>b~,\label{eq:thm:stop monotonicity:2}
                            \end{cases}
                 \end{align} for all $n\geq a+b+1$. The following subcases arise depending on the relative values of $a$ and $b$.
                 \begin{enumerate}[parsep=3pt]
                     \item[(i)] If $a+1\leq b$, then $R_U=a-n<a+1-n=R_V$. 
                     \item[(ii)] If $a=b~(\geq 1)$, then $V_{\nu+1}=a+1$ and $V_\nu=b$ and consequently, $V_{\nu-1}=1$. In the case where $b=1$, we have $V_{\nu-1}\geq V_\nu$, and by Observation~\ref{obs: increasing MPS}, it follows that $\nu=1$. Otherwise, $V_{\nu-2}=b-1\geq 1=V_{\nu-1}$ and which implies $\nu=2$. In both cases, we have $V_0\leq b$ and therefore, \begin{align*}
                          R_V & = a+b+1-n-V_0\\
                              &\geq a-n + b+1-b\\
                              & \geq a-n = R_U.     
                     \end{align*}
                 \item[(iii)] If $a>b$, then we must compare $U_0$ and $V_0$ in order to compare $R_U$ and $R_V$. We know $V_{\nu+1}=a+1=U_{\mu+1}+1$ and $V_\nu=b=U_\mu$. Thus, by Proposition~\ref{prop: new old MPS relation}, for all $0\leq k\leq \min\Set{\nu,\mu}$ we have,  
                 \begin{align}
                     V_{\nu - k} = U_{\mu-k}+\bar{P}_k \label{eq:thm:stop monotonicity:3}.
                 \end{align}
                 Now, by Lemma~\ref{lem: position of b}, $\nu=2\widehat{\nu}$ and $\mu=2\widehat{\mu}$ where, 
                 $$\widehat{\mu}= \mathrm{min}\left\{i\geq 0:\frac{a}{b}\leq O_i\right\},\; \widehat{\nu}=\mathrm{min}\left\{i\geq 0:\frac{a+1}{b}\leq O_i\right\}.$$
                 Recall that the sequence $\left(O_i\right)_{i\geq 0}$, where $O_i=P_{2i+2}/P_{2i+1}$, is increasing. Therefore, $\widehat{\mu}\leq \widehat{\nu}$ and consequently, $\mu\leq \nu$. Thus, we have,
                 \begin{align}
                     R_U-R_V &= V_0-1-U_0\tag{by Eq~\eqref{eq:thm:stop monotonicity:1}-\eqref{eq:thm:stop monotonicity:2}}\\
                         &= V_0-1-\left(V_{\nu-\mu}-\bar{P}_\mu\right) \tag{by Eq~\eqref{eq:thm:stop monotonicity:3}}\\
                         &= V_0 - V_{2(\widehat\nu-\widehat\mu)}+\bar{P}_{2\widehat\mu} -1 \tag{$\mu=2\widehat\mu,\; \nu=2\widehat\nu$}\\
                        &\leq V_0-V_0-P_{2\widehat{\mu}}-1\label{eq:thm:stop monotonicity:4} \\
                        &< 0 \tag{$P_{2\widehat{\mu}}\geq 0$}
                 \end{align}
                 Equation~\eqref{eq:thm:stop monotonicity:4} follows using the following facts:
                 \begin{itemize}
                     \item if $\widehat\nu=\widehat\mu$, then $V_{2(\widehat\nu-\widehat\mu)}=V_0$;
                     \item if $\widehat\nu>\widehat\mu$, then by Observation~\ref{obs: increasing MPS}(\ref{item:obs: increasing MPS:1}), $V_{2(\widehat\nu-\widehat\mu)}\geq V_2>V_0$;
                     \item $\bar{P}_{2\widehat\mu}=-P_{2\widehat\mu}$ by Definition~\ref{def: APS}.
                 \end{itemize}
                 

                 \end{enumerate}
        \end{enumerate}
        This concludes the proof.
\end{proof}

Let us restate this result as we often will use it. 
\begin{corollary}\label{cor:stop monotone twice}
    Consider $n,a,b\in \Nat$. Then, for all $n\geq a+b $, $\Rs(n;a,b)^* \leq \Rs(n;a+1,b-1)^*$.
\end{corollary}
\begin{proof}
     By Theorem~\ref{thm: stop monotonicity}, we have $\Rs(n;a,b)^* \leq \Rs(n;a+1,b)^*\leq \Rs(n;a+1,b-1)^*$.
\end{proof}

We understand the stops of {\sc Little John}. 
Next, we will show that the stops of {\sc Little John} and {\sc Robin Hood} are the same.

\section{Little John guides Robin Hood}\label{sec:stops}
For large heap sizes, {\sc Robin Hood} resembles {\sc Little John}. Robin Hood is wise when he listens to Little John. 

\begin{theorem}[Wise Robin Hood]\label{thm: stop dependency} 
Consider $n,a,b\in \mathbb{N}_0$. Then, for all $n\geq a+b$, the stops of $(n;a,b)$ are the same as the stops of $(n;a,b)^*$.

\end{theorem}

\begin{proof}
    We prove the statement only for the Right stops, as, for any game $G$, $\Ls(G)=-R(-G)$.
    
    We prove this using induction. Before initiating the induction steps, we first verify the statement for the cases where at least one of \(a\) or \(b\) is zero. Suppose $n\geq 0$.
    \begin{enumerate}
        \item If $a=0=b$, then $(n;a,b)=0=(n;a,b)^*$;
        \item If $a>0=b$, then $(n;a,b)=n=(n;a,b)^*$; 
        \item If $b>0=a$, then $(n;a,b)=n=(n;a,b)^*$.
    \end{enumerate}
    In all these cases, the stops of both games are equal as their game values are equal. Since the choice of $n$ was arbitrary, the statement holds true for all $n\geq 0$ in all these cases. 

    We now proceed to the induction on $a+b$. The base case of induction, \(a + b = 1\), is already proven.
    
    Suppose $n\geq a+b$. To proceed, we consider three cases based on the relative values of \(a\) and \(b\). Note that the scenarios where \(a = 0\), \(b = 0\), or both have already been resolved. Therefore, we now focus solely on cases where \(a, b > 0\). 
    \begin{enumerate}[wide, labelindent=0pt]
        \item If $a\leq b$, then, $\Rs(n;a,b)^*=-(n-a)$ by Lemma~\ref{lem: stop value}. Next, we compute $\Rs(n;a,b)$ to verify equivalence. By Theorem~\ref{thm: RH positions}(\ref{item:thm: RH positions:5}) and domination,
        \begin{align*}
            \Rs(n;a,b) &= \min_{i\in [a]} ~\Ls(n-i;a-i,b)
        \end{align*}
        We know, by Lemma~\ref{lem: stop value}(\ref{item:lem: stop value:2}) and Proposition~\ref{prop: relation between stop and game}(\ref{item:prop: relation between stop and game:1}), $\Ls(n-a;0,b) =-(n-a)$. Moreover, for all $1\le i\leq a-1$, we have,
        \begin{align*}
            \Ls(n-i;a-i,b) &= \Ls(n-i;a-i,b)^* \tag{by induction}\\
                         &= \Rs(n-a;a-i,b-(a-i))^*\\
                         &\geq -(n-a) \tag{by Prop~\ref{prop: -heap<=stop,LJ<=heap}(\ref{item:prop: -heap<=stop,LJ<=heap:2})}
        \end{align*}
         
        Hence, $\min_{i\in [a]} ~\Ls(n-i;a-i,b)=-(n-a)$. This completes the proof of this case.
    
    \item If $b\leq \frac{a}{2}$, then
    \begin{align}
        \Rs(n;a,b) &= \min_{i\in[b]} ~\Ls(n-i;a-i,b) \tag{by Thm~\ref{thm: RH positions}(\ref{item:thm: RH positions:5})}\\
                          &=\min_{i\in[b]} ~\Ls(n-i;a-i,b)^* \tag{by induction}\\
                          &=\min_{i\in[b]}~ -\Rs(n-i;b,a-i)^* \tag{by Prop~\ref{prop: relation between stop and game}(\ref{item:prop: relation between stop and game:1})}\\
                          &=\min_{i\in[b]}~ (n-i-b) \label{eq:thm: stop dependency:1}\\
                          &= n-2b.\nonumber
    \end{align}
    Equation~\eqref{eq:thm: stop dependency:1} follows using Lemma~\ref{lem: stop value}(\ref{item:lem: stop value:3}) as $a-i\geq b$ and $n\geq a+b$.
    
    Now, we compute $\Rs(n;a,b)^*$. 
    \begin{align}
        \Rs(n;a,b)^* &= \Ls(n-b;a-b,b)^* \tag{by Lem~\ref{lem: Wise RH is hot}(\ref{item:lem: Wise RH is hot:2})}\\
                    &= -\Rs(n-b;b,a-b)^* \tag{by Prop~\ref{prop: relation between stop and game}(\ref{item:prop: relation between stop and game:1})}\\
                    &= n-2b. \label{eq:thm: stop dependency:2}                  
    \end{align}
    Equation~\eqref{eq:thm: stop dependency:2} holds using Lemma~\ref{lem: stop value}(\ref{item:lem: stop value:3}) as $a-b\geq b$ and $n\geq a+b$. This concludes the proof for this case.

    \item If $\frac{a}{2} < b < a$, then, 
    \begin{align*}
        \Rs(n;a,b) &= \min_{i\in[b]}~ \Ls(n-i;a-i,b)\\
                      &=\min_{i\in[b]} ~-\Rs(n-i;b,a-i)^* \tag{by induction and Prop~\ref{prop: relation between stop and game}(\ref{item:prop: relation between stop and game:1})}
    \end{align*}
    Now, we divide the range of $i$ in two parts, in one, $a-i\geq b$ and in the other, $a-i<b$. So, we define $A\coloneqq \{i\in[b]: a-i\geq b\}$ and $B\coloneqq [b]\setminus A$. Note that $B$ cannot be empty by the assumption. Thus, $A=\{1,\ldots , a-b\}$ and $B=\{a-b+1,\ldots , b\}$. Then, 
    \begin{align*}
        \Rs(n;a,b) &= \min\left\{ \min_{i\in A}~ \left(-\Rs(n-i;b,a-i)^*\right), ~~\min_{i \in B}~\left(-\Rs(n-i;b,a-i)^*\right)\right\}\\
                 &= \min\left\{ \min_{i\in A}~ (n-i-b),~~ \min_{i \in B}~\left(-\Ls\left(n-a;b-(a-i),a-i\right)^*\right)\right\}\\ 
                 &= \min\left\{ n-(a-b)-b, ~~\min_{i \in B}~ \Rs\left(n-a;a-i,b-(a-i)\right)^*\right\} \tag{by definition of $A$}\\ 
                 &= \min\left\{ n-a, ~~\Rs\left(n-a;a-b,b-(a-b)\right)^*\right\} \tag{by Thm~\ref{thm: stop monotonicity}}\\
                 &= \Rs\left(n-a;a-b,b-(a-b)\right)^* \tag{by Prop~\ref{prop: -heap<=stop,LJ<=heap}(\ref{item:prop: -heap<=stop,LJ<=heap:2})}\\
                 &=\Ls(n-b;a-b,b)^* \\ 
                 &=\Rs(n;a,b)^*.
    \end{align*}
    The last two equalities holds using Lemma~\ref{lem: Wise RH is hot}(\ref{item:lem: Wise RH is hot:2}) and the condition $a/2<b<a$.  
    \end{enumerate}
    Since, the choice of $n$ was arbitrary from the set $\Set{a+b, a+b+1, \dots}$, all these cases holds for all $n\geq a+b$. This completes the proof.
\end{proof}

The next results focus on the geometric aspects of $\T(n;a,b)^*$. As we indicated in the Introduction, the typical behavior will depend of the `wealth ratio' $a/b$. By convention we choose $a\ge b$. Therefore the Left option will be (trivial) a Number, and all efforts will concern the Right options. Let the {\em wealth ratio} of a Right option be $w_b:=(a-b)/b$. 

\begin{theorem}[Little John  Thermographs]\label{thm: LJthermograph}
    For fixed integers $a, b > 0$, let $G = (n; a, b)^*$. Then, for $n$ sufficiently large, 
    \begin{enumerate}
        \item $G\in \LST$, if $\frac{a}{b} > \phi$;
        \item $G\in \RST$, if $\frac{a}{b} < \phi^{-1}$;
        \item $G\in \DT$, otherwise.
    \end{enumerate}
\end{theorem}

\begin{proof}
    We induct on $a+b$. Without loss of generality, consider $a\geq b>0$, as the thermograph of $(n;a,b)^*$ is the mirror image of $(n;b,a)^*$. 

    \emph{Base case 1:} If $a+b=2$, then $a/b=1$ and $G=(n;1,1)^*=\{n-1\mid 1-n\}$. Hence, $G\in \DT$ for all $n\geq 2$, as in Figure~\ref{fig: thm: LJthermograph:Base case 1}. 

    \begin{figure}[ht]
        \centering
        \begin{subfigure}[b]{0.49\textwidth}
            \centering
            \begin{tikzpicture}[>=stealth, scale=1,dot/.style = {circle, fill=red, minimum size=#1, inner sep=0pt, outer sep = 0pt}, dot/.default = 6pt]
                \draw [<->] (-2,0) -- (2,0); 
                \node[dot,scale=0.3,label=below:{\midsize $n-1$}] (a) at (-1,0) { };
                \node[dot,scale=0.3,label=below:{\midsize $1-n$}] (b) at (1,0) { };
                \node[dot, scale=0.3] (c) at (0,1) { };
                \node [below] at (0,0)   {\midsize{$0$}};
                \draw[black] (0,0.04)--(0,-0.04);
                \draw [red]  (-1,0) -- (0,1) -- (1,0);
                \draw [red]  (0,1) -- (0,1.5);
            \end{tikzpicture}
        \caption{The thermograph of $(n;1,1)^*$.}
        \label{fig: thm: LJthermograph:Base case 1}
        \end{subfigure}%
        \hspace{-1em}%
        \begin{subfigure}[b]{0.49\textwidth}
            \centering
            \begin{tikzpicture}[>=stealth, scale=1,,dot/.style = {circle, fill=red, minimum size=#1, inner sep=0pt, outer sep = 0pt}, dot/.default = 6pt]
             \draw [<->] (-2,0) -- (2,0); 
                   \coordinate (a) at (-1,0);
                  \coordinate (b) at (0,0);
                   \node [dot,scale=0.3,label=below:{\midsize $n-2$}] (b) at (0,0)   { };
                   \node [dot,scale=0.3,label=below:{\midsize $n-1$}] (a) at (-1,0)   { };
                   \node[dot,scale=0.3,label=right:{\midsize $1$}] (c) at (0,1) { };
                   \node [below] at (0.9,0) {\midsize{$0$}};
                  \draw [red]  (-1,0) -- (0,1) -- (0,0);
                  \draw [red]  (0,0) -- (0,1.5);
                  \draw[black] (0.9,0.04)--(0.9,-0.04);
        \end{tikzpicture}
        \caption{The thermograph of $(n;2,1)^*$.}
        \label{fig: thm: LJthermograph:base case 2}
        \end{subfigure}
        \caption{Thermographs of $(n;a,b)^*$ for small $a$ and $b$.}
    \end{figure}
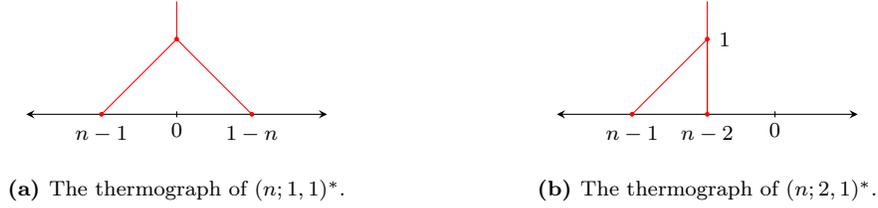

    \emph{Base case 2:} If $a+b=3$, then $a/b=2$ and  $G=(n;2,1)^*=\{n-1\mid\{n-2\mid 2-n\}\}$. Thus, for all $n\geq 3$, $G \in \LST$, as in Figure~\ref{fig: thm: LJthermograph:base case 2}.    
 
     Since $a\geq b$, for all $n\geq b$, the Left option of $G=(n;a,b)^*$ is the Number $n-b$, and thus its thermograph is a mast at $n-b$. Therefore, the contribution of $\T(G^L)$ to $\T(G)$ is a left wall of slope $-1$. 
     
     Next we analyze the contributions of the Right option for all sufficiently large $n$. This will require an induction argument, and we have analyzed the base cases above. Suppose that, for sufficiently large $n$, the statement holds for the thermographs of the options of $G=(n;a,b)^*$. To prove that the statement holds for the thermograph of $G$, we take different cases based on the ratio of the players' wealths.

    \begin{enumerate}[wide, labelindent=0pt]
    \item[\bf{(1)}] \textbf{\bm{$a/b>\phi$}:}  In this case $G^R=(n-b;a-b,b)^*.$ We must prove that, for sufficiently large $n$,  $G\in \LST$. The thermograph of $G^R$ depends on the wealth ratio $w_b=(a-b)/b$. Since $a/b>\phi$, then $(a-b)/b>\phi^{-1}$. Hence, by induction, $G^R\in \DT \cup \LST$. Thus, $\lw(G^R)$ has slope $-1$. In either case, since we must prove that $\RW(G)$ is a vertical line, we must verify that the contribution to $\T(G)$ from $\T(G^L)$ meets the contribution from small left wall of $\T(G^R)$.

        \begin{enumerate}[wide, parsep=4pt, topsep=6pt, labelindent=1pt]
            \item[\bf{(1A)}] \bm{$w_b>\phi$}\textbf{:} In this case, $G^R\in \LST$ (see Figure~\ref{fig: thermograph of G in case 1A}). The small left wall of $\T(G^R)$ ends at the temperature  $t(G^R)=\Ls(G^R)-\Rs(G^R)$. Recall that the contribution from $\T(G^L)$ to $\T(G)$ is a small left wall with slope $-1$. It intersects the left rotated$\lw(G^R)$ 
            if and only if 
               \begin{align}
            \Ls(G)-\Rs(G)&=\notag\\
                \Ls(G)-\Ls(G^R)&\leq \Ls(G^R)-\Rs(G^R),\label{eq: LJ Therm eq 1}
            \end{align}


            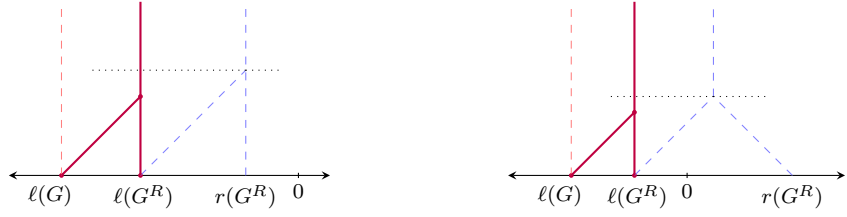
\begin{figure}[ht!]
                \centering
                \begin{subfigure}[b]{0.45\textwidth}
                    \centering
                    \begin{tikzpicture}[>=stealth, scale=.7,dot/.style = {circle, fill=red, minimum size=#1, inner sep=0pt, outer sep = 0pt}, dot/.default = 6pt]
                        \draw [<->] (-1.5,0) -- (4.6,0); 
                        \node[anchor=north] (l) at (-0.5,0) {\midsize $\Ls(G)\quad$};
                        \node[dot,purple,scale=0.3] at (l.north) {};
                        \node[anchor=north] (rl) at (1,0) {\midsize $\;\Ls(G^R)$};
                        \node[dot,purple,scale=0.3] at (rl.north) {};
                        \node[anchor=north] (rr) at (3,0) {\midsize $\Rs(G^R)$};
                        \node[anchor=north] (extra) at (4,0) {\midsize $0$};%
                        \draw[dashed, red!50] (l.north)-- +(0,3.3);
                        \draw[dashed,blue!50] (rl.north)-- +(2,2) node (e1) { };
                        \draw[dashed, blue!50] (rr.north)--+(0,3.3);
                        \draw[dotted,black] (e1.center)--+(-3,0) (e1.center)--+(0.7,0);
                        \draw[purple, thick] (l.north)-- +(1.5,1.5) (rl.north)--+(0,3.3);
                        \node[dot,purple,scale=0.3] at (1,1.5) {};
                        \draw[black] (4,0.05)--(extra.north)--(4,-0.05);       
                    \end{tikzpicture}
                    \caption{The wealth ratio of $G^R$ is \(\frac{a-b}{b} > \phi\).}
                    \label{fig: thermograph of G in case 1A}
                \end{subfigure}%
                \hspace{-0.5em}
                \begin{subfigure}[b]{0.45\textwidth}
                    \centering
                    \begin{tikzpicture}[>=stealth, scale=.7,dot/.style = {circle, fill=red, minimum size=#1, inner sep=0pt, outer sep = 0pt}, dot/.default = 6pt]
                        \draw [<->] (-.4,0) -- (6,0); 
                        \node[anchor=north] (l) at (0.8,0) {\midsize $\Ls(G)\quad$ };
                        \node[dot,purple,scale=0.3] at (l.north) {};
                        \node[anchor=north] (rl) at (2,0) {\midsize $\;\Ls(G^R)$};
                        \node[dot,purple,scale=0.3] at (rl.north) {};
                        \node[anchor=north] (rr) at (5,0) {\midsize $\Rs(G^R)$};
                        \node[anchor=north] (extra) at (3,0) {\midsize $0$};%
                        \draw[dashed, red!50] (l.north)-- +(0,3.3);
                        \draw[dashed,blue!50] (rl.north)-- +(1.5,1.5);
                        \draw[dashed, blue!50] (rr.north)--+(-1.5,1.5) node (e1) { };
                        \draw[dashed, blue!50] (e1.center)-- +(0,1.8);
                        \draw[dotted,black] (e1.center)--+(-2,0) (e1.center)--+(1,0);
                        \draw[purple, thick] (l.north)-- +(1.2,1.2) (rl.north)--+(0,3.3) node (e2) { };
                        \node[dot,purple,scale=0.3] at (2,1.2) {};
                        \draw[black] (3,0.05)--(extra.north)--(3,-0.05);       
                    \end{tikzpicture}  
                    \caption{The wealth ratio of $G^R$ is \(\phi^{-1}<\frac{a-b}{b} < \phi\).}
                    \label{fig: thermograph of G in case 1B}
                \end{subfigure}
                \caption{Thermographs of $G=(n;a,b)^*$ and its options when $\frac{a}{b}>\phi$. 
                In each figure, the Red dashed line represents \(\T(G^L)\), blue dashed lines represent \(\T(G^R)\), black dotted line indicates the temperature of \(G^R\) and \(\T(G)\) is given by solid purple lines.}
                \label{fig: Thermograph of G case 1}
            \end{figure}
            where the first equality holds since there is only one option. For the second inequality, we need to find $\Rs(G^R)$ and $\Ls(G^R)$. For this purpose, let $(U_i)_{i\geq 0}$ be an MP-sequence with $U_{\mu+1}=a-b$ and $U_\mu=b$ for some $\mu\geq 0$. Then, by Lemma~\ref{lem: stop value}, we have $\Ls(G)=n-b$, $\Ls(G^R)=n-2b$, and $\Rs(G^R)= n-b-(a-b+b)+U_0$ which implies $$\Ls(G)-\Ls(G^R)=b\; \text{ and }\;\Ls(G^R)-\Rs(G^R)=a-b-U_0.$$ Hence, by \eqref{eq: LJ Therm eq 1}, it suffices to prove that $U_0\le a-2b$. Note that $a-2b>0$ and therefore $U_{\mu-1}=a-2b$ and $\mu\geq 1$. Thus, the problem reduced to show that $U_0\leq U_{\mu-1} \text{, for all } \mu\geq 1.$

            If $\mu\geq 3$ or $\mu=1$, then, by Observation~\ref{obs: increasing MPS}(\ref{item:obs: increasing MPS:1}), $U_0\leq U_{\mu-1}$. If $\mu=2$, then $U_0=3b-a$ and $U_{\mu-1}=U_1=a-2b$. Now, by applying Observation~\ref{obs: increasing MPS}(\ref{item:obs: increasing MPS:1}), we get $3b-a\geq a-2b$. This implies $w_b\leq 1.5$, which is a contradiction to our assumption.

            \vspace{7pt}
            \item[\bf{(1B)}] \bm{$w_b<\phi$}\textbf{:} In this case, $G^R\in \DT$ (see Figure~\ref{fig: thermograph of G in case 1B}). Thus, $\lw(G^R)$ ends at the temperature $t(G^R)=\frac{1}{2} \left(\Ls(G^R)-\Rs(G^R)\right)$. Recall that the contribution from $\T(G^L)$ to $\T(G)$ is a small left wall of slope $-1$, and it meets the left rotated $\lw(G^R)$ if and only if 
            \begin{equation}\label{eq: LJ Therm eq 2}
                \Ls(G)-\Ls(G^R)\leq \frac{1}{2}\left(\Ls(G^R)-\Rs(G^R)\right).
            \end{equation} 
            By Lemma~\ref{lem: stop value}, $\Ls(G^R)-\Rs(G^R)=2n+c$ and $\Ls(G)-\Ls(G^R)=d$ for some constants $c$ and $d$, with respect to $n$. Hence, for all sufficiently large $n$, $\Ls(G)-\Ls(G^R)\leq \frac{1}{2}\left(\Ls(G^R)-\Rs(G^R)\right)$.


        \end{enumerate}

        \item[\bf{(2)}] \bm{$1< a/b< \phi$}\textbf{:} In this case, we must prove that, for all sufficiently large $n$, $G\in \DT$. Here $G=\{n-b\mid(n-b;a-b,b)^*\}$, and thus, by $\frac{a-b}{b}<\phi^{-1}$, by induction, for $n$ sufficiently large, $G^R\in \RST$ (see Figure~\ref{fig: Thermograph of G case 2}). Hence, the contribution to $\T(G)$ from $\T(G^R)$ is a small right wall with slope $+1$ and the contribution from $\T(G^L)$ is a small left wall with slope $-1$. Hence, for all sufficiently large $n$, $G\in \DT$ (as in Figure~\ref{fig: Thermograph of G case 2}). In this case, the lower bound for $n$ is $a+b$, by using Lemma~\ref{lem: stop value}.  

            \begin{figure}[ht]
                \centering
                \begin{subfigure}[b]{0.48\textwidth}
                \centering
                    \begin{tikzpicture}[scale=0.7,>=stealth,dot/.style = {circle, fill=red, minimum size=#1, inner sep=0pt, outer sep = 0pt}, dot/.default = 6pt]
                    \draw [<->] (-1,0) -- (6,0); 
                        \node[anchor=north] (l) at (0,0) {\midsize $\Ls(G)$};
                        \node[dot,purple,scale=0.3] at (l.north) {};
                        \node[anchor=north] (rl) at (3,0) {\midsize $\Ls(G^R)$};
                        \node[dot,purple,scale=0.3] at (rl.north) {};
                        \node[anchor=north] (rr) at (5,0) {\midsize $\Rs(G^R)$};
                        \node[anchor=north] (extra) at (1.8,0) {\midsize 0}; 
                        \draw[dashed, red!50] (l.north)-- +(0,3.3);
                        \draw[dashed,blue!50] (rl.north)-- +(0,3.3);
                        \draw[dashed, blue!50] (rr.north)--+(-2,2) node (e2) { };
                        \draw[purple, thick] (l.north)-- +(1.5,1.5) (rl.north)--+(-1.5,1.5) node (e1) { };
                        \node[dot,purple,scale=0.3] at (e1) {};
                        \draw[purple, thick] (e1.center)--+(0,1.8);
                        \draw[black] (extra.north)--+(0,0.05)--+(0,-0.1);
                \end{tikzpicture}
                \caption{The wealth ratio is $1<\frac{a}{b}<\phi$.}
                \label{fig: Thermograph of G case 2}
                \end{subfigure}%
                \hfill
                \begin{subfigure}[b]{0.48\textwidth}
                \centering
                    \begin{tikzpicture}[scale=0.7,>=stealth,dot/.style = {circle, fill=red, minimum size=#1, inner sep=0pt, outer sep = 0pt}, dot/.default = 6pt]
                    \draw [<->] (-1,0) -- (4.5,0); 
                        \node[anchor=north] (l) at (0,0) {\midsize $\Ls(G)$};
                        \node[dot,purple,scale=0.3] at (l.north) {};
                        \node[anchor=north] (rl) at (3,0) {\midsize $\Ls(G^R)$};
                        \node[dot,purple,scale=0.3] at (rl.north) {};
                        \node[anchor=north] (extra) at (1.8,0) {\midsize $0$};
                        \draw[dashed, red!50] (l.north)-- +(0,3.3);
                        \draw[dashed,blue!50] (rl.north)-- +(0,3.3);
                        \draw[purple, thick] (l.north)-- +(1.5,1.5) (rl.north)--+(-1.5,1.5) node (e1) { };
                        \node[dot,purple,scale=0.3] at (e1) {};
                        \draw[purple, thick] (e1.center)--+(0,1.5);
                        \draw[black] (extra.north)--+(0,0.05)--+(0,-0.1);
                \end{tikzpicture}
                \caption{The wealth ratio is $\frac{a}{b}=1$.}
                \label{fig: Thermograph of G case 3}
                \end{subfigure}
                \caption{Thermographs of $G$ when $\frac{a}{b}<\phi$. In each figure, the Red dashed line represents \(\T(G^L)\), blue dashed lines represent \(\T(G^R)\), and \(\T(G)\) is given by purple solid lines.}
            \end{figure}
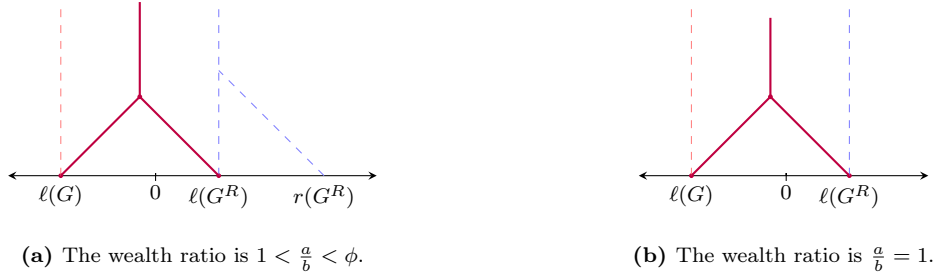

        \item[\bf{(3)}] \bm{$a/b=1$}\textbf{:} If $n>a=b$, obviously $G\in\DT$.
    \end{enumerate}
    Thus {\sc Little John}'s tent structures have been established.
\end{proof}
The term ``orthodox option'' is often used in the context of thermograph plots. Such options contribute to the thermograph. 
The key to the {\sc Robin Hood} thermographs depends on its orthodox link with {\sc Little John}. Later in  Theorem~\ref{thm:RHthermographs} we will exploit further the stop monotonicity of {\sc Robin Hood} options. 

\begin{lemma}[Options Stop Monotonicity]\label{lem: osm (option stop monoto)}
    Let $n,a,b>0$ be integers and consider the {\sc Robin Hood} game $(n;a,b)$, with Left and Right options $L_i=(n-i;a,b-i)$ and $R_i=(n-i;a-i,b)$, respectively, where $i\in[b]$. For all sufficiently large heap sizes $n$,
    \begin{enumerate}
        \item $\Ls(L_1)= \Ls(L_2)=\dots=  \Ls(L_b)=n-b$;
        \item $\Rs(L_1)\leq  \Rs(L_2)\leq \dots  \leq \Rs(L_b)=n-b$;
        \item $\Ls(R_1)\geq \Ls(R_2)\geq \dots\geq \Ls(R_b)$;
        \item $\Rs(R_1)\geq \Rs(R_2)\geq \dots\geq \Rs(R_b)$.
    \end{enumerate}
\end{lemma}
\begin{proof}
    Without loss of generality, consider $a\geq b$. Now, we compare the stops of the Left options of $G$ using {\sc Little John} Stop Monotonicity, Theorem~\ref{thm: stop monotonicity} and Corollary~\ref{cor:stop monotone twice}. We use the following results several times in the proof:
    \begin{itemize}
        \item By Theorem~\ref{thm: stop dependency}, the stops of Robin Hood and Little John are the same;
        \item Lemma~\ref{lem: stop value} gives the exact stop values.
    \end{itemize}
    
    For any fixed $i\in [b]$, we have:
    \begin{align}
        \Ls(L_i)& = \Ls(n-i;a,b-i)^*\nonumber\\
                &= \Rs(n-b;a,0)^*\\
                       & = n-b  \label{eq:osm 1}\\
        \Rs(L_i) & = \Rs(n-i;a,b-i)^* \nonumber\\
                       & =\begin{cases}
                           n-b & \text{ if }i=b;\\
                           \Ls(n-b;a-b+i,b-i)^* & \text{ otherwise.}
                       \end{cases} \label{eq:osm 2}
    \end{align}
    We have used that $b-i<a$ for all $i\in [b]$. Therefore, by Corollary~\ref{cor:stop monotone twice}, \begin{equation}\label{eq:osm 3}
        \Rs(L_1)\leq  \Rs(L_2)\leq \dots  \leq \Rs(L_{b-1}).
    \end{equation} And, by Proposition~\ref{prop: -heap<=stop,LJ<=heap}, 
    \begin{equation}\label{eq:osm 4}
        \Rs(L_{b-1})\leq n-b.
    \end{equation}
    Thus, equations~\eqref{eq:osm 1}-\eqref{eq:osm 4} together prove items (1) and (2).
    
    For item (3), we compare the Left stops of the Right options. By Theorem~\ref{thm: stop dependency}, for all $i\in[b]$, we have
        \begin{align}
            \Ls({R_i})& = \Ls(n-i;a-i,b)^* \\
                              & = \begin{cases}
                                   n-b-i, & \text{ if }a-i\geq b;\\
                                   \Rs(n-a;a-i,b-a+i)^*, & \text{ otherwise}.
                               	\end{cases}\label{eq:osm 0}
        \end{align}
        In Equation~\eqref{eq:osm 0}, the behavior of $\Ls({R_i})$ changes when $i$ increases such that $a-i<b$. So, we define $\alpha$ as the minimum natural number $i$ such that $a-i<b$, i.e., $\alpha\coloneqq \min\Set{i\in \Nat \SetSymbol a-i<b}$. Note that $\alpha=a-b+1$. Put differently, $\alpha$ is the fewest tokens Right needs to remove in $(n;a,b)$ to make Left's wealth less than Right's. But, if $\alpha>b$, then Right cannot do this. Therefore, let us first assume $\alpha\leq b$. Now, we rewrite Equation~\eqref{eq:osm 0} as:
        \begin{align}
            \Ls(R_i) & = \begin{cases}
                                   n-b-i, & \text{ if } i<\alpha;\\
                                   \Rs(n-a;a-i,b-a+i)^*, & \text{ if }i\geq \alpha.
                               	\end{cases}\label{eq:osm 5}
        \end{align}
        Then, by Corollary~\ref{cor:stop monotone twice}, we have 
        \begin{equation}\label{eq:osm 6}
            \Ls(R_\alpha)\geq \Ls(R_{\alpha+1})\geq \dots \geq \Ls(R_b).
        \end{equation}
        Now, we compare $\Ls(R_{\alpha-1})$ with $\Ls(R_\alpha)$, 
        \begin{align}
            \Ls(R_{\alpha-1})&=n-b-\alpha+1 \tag{by Eq~\eqref{eq:osm 5}}\\
                             &=n-a\nonumber\\
                            &\geq \Rs(n-a;a-\alpha,b-a+\alpha)^* \tag{by Prop~\ref{prop: -heap<=stop,LJ<=heap}}\\
                            &=\Ls(R_\alpha)\label{eq:osm 7}
        \end{align}
        Equations~\eqref{eq:osm 5}-\eqref{eq:osm 7} prove item (3) when $\alpha\leq b$. If $\alpha>b$, then the proof of item 3 is complete because $\Ls(R_i)=n-b-i$ for all $i\in [b]$. 
        
    
    For item (4), we compare the Right stops of the Right options of $(n;a,b)$.
    \begin{align}
        \Rs(R_i)&=\Rs(n-i;a-i,b)^*\\
                  &=\begin{cases}
                      \Ls(n-i-b;a-i-b,b)^*, & \text{ if }a-i>b;\\
                      -(n-a), & \text{ otherwise}.
                  \end{cases}\\
                  &=\begin{cases}
                          n-i-2b, &\text{ if }a-i\geq 2b;\\
                          \Rs(n-a;a-i-b,2b-a+i)^*, & \text{ if } b<a-i<2b; \\
                      -(n-a), & \text{ if }a-i\leq b.
                  \end{cases}\label{eq:osm 9}
    \end{align}
    Define $\beta\coloneqq\min\Set{i\in\Nat\SetSymbol a-i< 2b}$ and $\gamma\coloneqq\min\Set{i\in\Nat\SetSymbol a-i\leq b}$. Note that $\beta=a-2b+1$ and $\gamma=a-b$. If $\gamma\leq b$, then we rewrite Equation~\eqref{eq:osm 9} as:
    \begin{align}
        \Rs(R_i)&=\begin{cases}
                          n-i-2b, &\text{ if }i<\beta;\\
                          \Rs(n-a;a-i-b,2b-a+i)^*, & \text{ if } \beta\leq i<\gamma;\\
                      -(n-a), & \text{ if }\gamma \leq i\leq b.
                  \end{cases}
    \end{align} Thus, we have    
    \begin{align}
        \Rs(R_1)\geq \Rs(R_2)\geq \dots\geq \Rs(R_{\beta-1}),\label{eq:osm 10}\\
        \Rs(R_\beta)\geq \Rs(R_{\beta+1})\geq \dots\geq \Rs(R_{\gamma-1}),\label{eq:osm 8}\\
        \Rs(R_\gamma)\geq \Rs(R_{\gamma+1})\geq \dots\geq \Rs(R_{b}),\label{eq:osm 13}
    \end{align}
    where Equation~\eqref{eq:osm 8} follows by Corollary~\ref{cor:stop  monotone twice}. Now, by Proposition~\ref{prop: -heap<=stop,RH<=heap}, 
    \begin{align}
        \Rs(R_{\beta-1}) &= (n-a)\geq \Rs(R_{\beta}),\label{eq:osm 11}\\
        \Rs(R_{\gamma-1})&\geq -(n-a)=\Rs(R_{\gamma}).\label{eq:osm 12}
    \end{align}
    Equations~\eqref{eq:osm 10}-\eqref{eq:osm 12} complete the proof when $\gamma\leq b$. If $\gamma>b$ and $\beta\leq b$, then Equations~\eqref{eq:osm 10},\eqref{eq:osm 8} and \eqref{eq:osm 11} complete the proof. The proof of case when $\beta>b$ follows by Equation~\eqref{eq:osm 10}.
\end{proof}


\section{A solution for the Pingala Era wetland tribes}\label{sec:main}
The toolbox is now complete, and we arrive at the main theorem, in terms of thermographs. At last, in this section, we revisit Theorem~\ref{thm: main theorem} by including a short proof, interpreting {\sc Robin Hood}'s thermographs in terms of mean values and temperatures. To simplify reading the proof we locally abbreviate some of our standard notation.

\begin{notation}
     Consider a {\sc Robin Hood} game $G=(n;a,b)$. For $ i \in [b]$, let $L^i=(n-i;\, a,b-i)$ represent the Left options and let $R_i=(n-i;\, a-i,b)$ represent the Right options. 
    The reason for the super- and sub-scripts is the following short hand notation for the stops of these options:
    \begin{itemize}
        \item Let $\Rs^i$ and $\Rs_i$ denote the Right stops of $L^i$ and $R_i$, respectively, and let $\Ls^i$ and $\Ls_i$ denote the Left stops of $L^i$ and $R_i$, respectively;
        \item Let $\RW^i$ and $\RW_i$ denote the large right walls of $L^i$ and $R_i$, respectively, and let $\LW^i$ and $\LW_i$ denote the large left walls of $L^i$ and $R_i$, respectively;
        \item Similarly, let $\rw^i$, $\rw_i$, $\lw^i$ and $\lw_i$ denote the small walls.
    \end{itemize}
    This use of sub- and super-scripts can be generalized to any function on $R_i$ and $L^i$ respectively; for example $m(R_b)=m_b$ and $m(L^a)=m^a$, etc. 
    Moreover, when the options are penalized by $p$, we write $\Rs^i(p)$ and $\Ls^i(p)$ for the Right and Left stops of $p$-penalized $L^i$, respectively, and $\Rs_i(p)$ and $\Ls_i(p)$ for the Right and Left stops of $R_i$ penalized by $p$, respectively. 
    Aligning with these notations, denote a typical $G^R\in\DT$ by \(R_\delta\), with Left and Right stops \(\Ls_\delta\) and \(\Rs_\delta\), respectively, and denote a typical $G^L\in\DT$  by \(L^\delta\), with Left and Right stops \(\Ls^\delta\) and \(\Rs^\delta\), respectively. Similarly, denote \(G^R\in\ST\) by \(R_\sigma\), with stops \(\Ls_\sigma\) and \(\Rs_\sigma\), and denote \(G^L\in\ST\) by \(L^\sigma\), with stops \(\Ls^\sigma\) and \(\Rs^\sigma\). 
\end{notation}

We make use of a partial order of large left and right walls. 

\begin{definition}[Wall Partial Order]
     Let $f,g: \mathbb{D^+} \rightarrow \mathbb{D}$, and let $F=\{(f(y),y)\mid y \in\mathbb{D^+}\}$ and $G=\Set{(g(y),y)\SetSymbol y \in\mathbb{D}^+}$. Then $F\ge G$ if, for all $y\in\mathbb{D^+}$, $f(y)\ge g(y)$.
\end{definition}
Thus, for example $\LW_i\ge \LW_j$ if, for all $p\in \mathbb{D^+} $, $\Ls_i(p)\ge \Ls_j(p)$. We will see that {\sc Little John} and {\sc Robin Hood} have the same mean values and temperatures for large heaps, and the reason for that is that they have the same thermographs.
\begin{theorem}[Robin Hood Thermographs]\label{thm:RHthermographs}
    Let  $a,b\geq 0$ be integers. Then, for any sufficiently large heap size $n$, $\T(n;a,b)=\T(n;a,b)^*$. 
\end{theorem}
\begin{proof}
    Let $G=(n;a,b)$ and let $H=(n,a,b)^*$. Our goal is to demonstrate that, for $n$ sufficiently large, the thermographs of $H$ and $G$ are identical. Thus, it suffices to show that, for large $n$, the thermograph of $G$ solely depends on the Little John options.
    
    We induct on $a+b$. Without loss of generality, consider $a\geq b$, as $\T(n;a,b)$ is the mirror image of $\T(n,b,a)$. 
    
    If $a=b=0$, then $G=0=H$ and if $a>b=0$, then $G=n=H$. 
    Thus, the result holds in these cases, so suppose both $a$ and $b$ are non-zero.
    
    If $a+b=2$, then $(n;1,1)=\{(n-1)\mid-(n-1)\}=(n;1,1)^*$. Hence, the thermographs of $G$ and $H$ are the same. 

    Now, suppose that the statement holds for the options of $G$, for all sufficiently large heap sizes $n$. 
        

    We begin by examining the Left options of $G$ and their corresponding thermographs.
    By Option Stop Monotonicity, Lemma~\ref{lem: osm (option stop monoto)}, we have
    \begin{equation}
        \Ls^1=\Ls^2=\dots=\Ls^b=n-b=\Rs^b\geq \Rs^{b-1}\geq \dots \geq \Rs^1.
    \end{equation}
    This implies that $L^b$ is the option with the Largest Left stop and the thermograph of $L^b$ is a mast at $n-b$. Hence, by Lemma~\ref{lem: MvsO mast vs options}, the left wall of $\T(G)$ does not depend on any Left option other than $L^b$, which is the Little John option. 

    For the rest of the proof, we demonstrate that, for large $n$, $\RW(G)$ does not depend on any Right option other than the Little John option $R_b$.  That is, we prove that, for large $n$, for all $i$, $\LW_b\le\LW_i$. 
    By induction, the thermographs of the Right options are the same as those of the corresponding {\sc Little John} thermographs, and consequently they depend on the option's wealth ratio $w_i=\frac{a-i}{b}$, for $i\in[b]$. Recalling Theorem~\ref{thm: LJthermograph}, we have $R_i\in \RST$ if $w_i<\phi^{-1}$, $R_i\in \DT$ if $\phi^{-1}<w_i<\phi$, and $R_i\in \LST$, otherwise.
    
    We take different cases based on $G$'s wealth ratio $\frac{a}{b}$.

    \begin{enumerate}[wide, labelindent=0pt, topsep=5pt, parsep=3pt]
    \item[\bf{(1)}] $\bm{\frac{a}{b}> \phi:}$
        By $i\in [b]$, the wealth ratios, \( w_i=\frac{a-i}{b} \), of the Right options fall within the interval \( (\phi^{-1}, \frac{a}{b}) \). Hence, by induction, their thermographs are either $\DT$ or $\LST$. When the ratio  \(w_i\) drops below \( \phi \), the Right options' thermographs change from $\LST$ to $\DT$. Therefore, we define \( \alpha \) as the smallest Right removal ($i$) for which $R_i\in \DT$, i.e.,
        \[
        \alpha = \min \left\{i \geq 1 : w_i < \phi\right\}.
        \]
    
        By this definition, if \( \alpha > b \), then the thermographs of all Right options are left single tents. Otherwise, at least one of the thermographs is a double tent. So, we take sub-cases based on $\alpha$.
        
        Recall, by Lemma~\ref{lem: osm (option stop monoto)},  
        \begin{align}
            \Ls_1&\geq \Ls_2 \geq \dots\geq \Ls_b,\text{ and}\label{eq: decreasing LGRi}\\
            \Rs_1&\geq\Rs_2\geq \dots\geq\Rs_b\label{eq: decreasing RGRi}.
        \end{align}
        \begin{enumerate}[wide, labelindent=0pt, topsep=5pt, parsep=3pt]
            \item[\bf{(1A)}] $\bm{\alpha>b\!:}$
            Let $R_\sigma\in \LST$ be any Right option other than $R_b\in\LST$. We must prove that $\LW_b\le \LW_\sigma$.  
            Clearly $\Ls_b\le \Ls_\sigma$, by the monotonicity property in Equation~\eqref{eq: decreasing LGRi}, so $\LW_b\not>\LW_\sigma$. 
            
            Assume, for a contradiction, that $\LW_b\not\le\LW_\sigma$. Since $\Ls_b\le \Ls_\sigma$, this assumption implies that $\LW_\sigma$ crosses $\LW_b$ at some point, that is, there is a penalty $p$ such that $\Ls_b(p)>\Ls_\sigma(p)$. The only way this can occur is if $\lw_\sigma$ intersects $\LW_b$ 
            above $t_b$ (see Figure~\ref{fig: Thermograph of Right options in case 1A contradiction}). In this case, since both $R_b,R_\sigma\in \LST$, we must have $\RW_\sigma < \RW_b$. This would imply $\Rs_\sigma < \Rs_b$, which contradicts the monotonicity property in Equation~\eqref{eq: decreasing RGRi}. Hence, we are left with the situations in Figures~\ref{fig: Thermograph of Right options in case 1A1} and \ref{fig: Thermograph of Right options in case 1A2} that both confirm that 
            $\LW_b\le \LW_\sigma$. 

            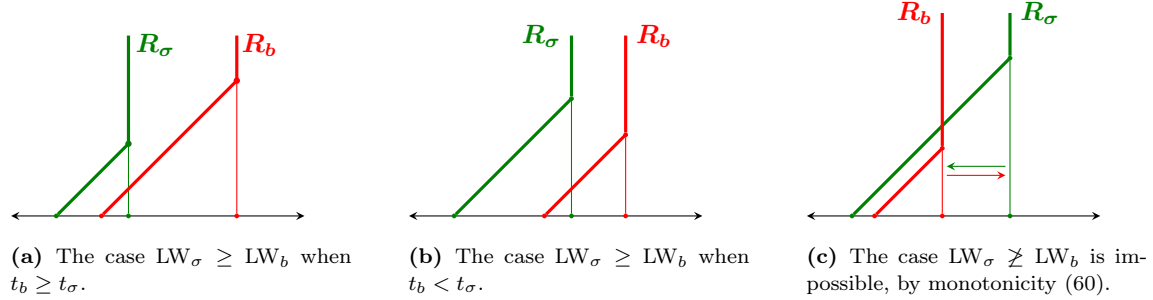
\begin{figure}[ht!]
		\centering
	 	\begin{subfigure}[b]{0.3\textwidth} 
                \begin{tikzpicture}[>=stealth, scale=0.6, dot/.style = {circle, fill, minimum size=#1, inner sep=0pt, outer sep = 0pt}, dot/.default = 6pt]  
                    \draw [<->] (-1,0) -- (5.5,0); 
                    \node[dot, fill=green!50!black, scale=0.3] (l1) at (0,0) { };
                    \node[dot, scale=0.3, fill=green!50!black] (r1) at (1.6,0) { };
                    \node[dot, scale=0.3, fill=red] (l2) at (1,0) { };
                    \node[dot, scale=0.3, fill=red] (r2) at (4,0) { };
                    \draw[green!50!black, very thick] (l1)--+(1.6,1.6) node[dot, scale=0.4] (C) { };
                    \draw[green!50!black] (r1)--+(0,1.6);
                    \draw[green!50!black, very thick] (1.6,4)--(1.6,1.6);
                    \draw[red, very thick] (l2)--+(3,3) node[dot, scale=0.4] (G) { };
                    \draw[red] (r2)--+(0,3);
                    \draw[red, very thick] (4,3)--(4,4);
                    \node[color=green!50!black] (rb) at (2.2,3.8) {\bm{$R_\sigma$}};
                    \node[color=red] (rb) at (4.5,3.8) {\bm{$R_b$}};
		    \end{tikzpicture}
                \caption{The case $\LW_\sigma\ge\LW_b$ when $t_b\ge t_\sigma$.}
                \label{fig: Thermograph of Right options in case 1A1}
	 	\end{subfigure}%
			\hfill
            \begin{subfigure}[b]{0.3\textwidth}
                \begin{tikzpicture}[>=stealth, scale=0.6, dot/.style = {circle, fill, minimum size=#1, inner sep=0pt, outer sep = 0pt}, dot/.default = 6pt]  
                    \draw [<->] (-1,0) -- (5.5,0); 
                    \node[dot, fill=green!50!black, scale=0.3] (l1) at (0,0) { };
                    \node[dot, scale=0.3, fill=green!50!black] (r1) at (2.6,0) { };
                    \node[dot, scale=0.3, fill=red] (l2) at (2,0) { };
                    \node[dot, scale=0.3, fill=red,] (r2) at (3.8,0) { };
                    \draw[green!50!black, very thick] (l1)--+(2.6,2.6) node[dot, scale=0.3] (C) { };
                    \draw[green!50!black] (C)--(r1);
                    \draw[green!50!black, very thick] (C)--+(0,1.4) node[anchor=east] (D) {\bm{$R_\sigma$}};
                    \draw[red, very thick] (l2)--+(1.8,1.8) node[dot, scale=0.3] (G) { };
                    \draw[red] (G)--(r2);
                    \draw[red, very thick] (G)--+(0,2.2) node[anchor=west] (H) {\bm{$R_b$}};
		    \end{tikzpicture}
                \caption{The case $\LW_\sigma\ge\LW_b$ when $t_b<t_\sigma$.}
                \label{fig: Thermograph of Right options in case 1A2}
            \end{subfigure}%
                \hfill 
            \begin{subfigure}[b]{0.3\textwidth} 
                \begin{tikzpicture}[>=stealth, scale=0.6, dot/.style = {circle, fill, minimum size=#1, inner sep=0pt, outer sep = 0pt}, dot/.default = 6pt]  
                    \draw [<->] (-1,0) -- (5.5,0); 
                    \node[dot, fill=green!50!black, scale=0.3] (l1) at (0,0) { };
                    \node[dot, scale=0.3, fill=green!50!black] (r1) at (3.5,0) { };
                    \node[dot, scale=0.3, fill=red] (l2) at (0.5,0) { };
                    \node[dot, scale=0.3, fill=red] (r2) at (2,0) { };
                    \draw[green!50!black, very thick] (l1)--+(3.5,3.5) node[dot, scale=0.3] (C) { };
                    \draw[green!50!black, very thick] (C)--+(0,1) node[anchor=west] (D) {\bm{$R_\sigma$}};
                    \draw[green!50!black] (r1)--(C); 
                    \draw[red, very thick] (l2)--+(1.5,1.5) node[dot, scale=0.3] (G) { };
                    \draw[red, very thick] (G)--+(0,3) node[anchor=east] (H) {\bm{$R_b$}};
                    \draw[red] (r2)--(G);
                    \draw [<-, red] (3.4,.9) -- (2.1,.9);
                    \draw [->, green!50!black] (3.4,1.1) -- (2.1,1.1);
                    
		    \end{tikzpicture}
                \caption{The case $\LW_\sigma\not\geq\LW_b$ is impossible, by monotonicity \eqref{eq: decreasing RGRi}.}
                \label{fig: Thermograph of Right options in case 1A contradiction}
	 	\end{subfigure}%
            \caption{Thermographs of Right options of $G$ in case~(1A) when both $R_b, R_\delta\in \RST$.}
            \label{fig: Thermograph of Right options in case 1A}
	\end{figure}

        \item[\bf{(1B)}] {$\bm{\alpha\leq b\!:}$} 
        If $\alpha\leq i\leq b$, then by induction and definition of $\alpha$, $R_i\in\DT$. Otherwise, if $i<\alpha$, then $R_i\in\LST$. Therefore, in this case, we must compare $R_b\in\DT$ with both $R_\sigma\in\LST$ and $R_\delta\in \DT$.

        We start with $R_b\in \DT$ and $R_\sigma\in\LST$.\\ 
        
        \noindent {\bf Claim:} If $R_b\in \DT$ and $R_\sigma\in\LST$, then, for sufficiently large heap sizes,  $m_b\leq m_\sigma$.
        
        \noindent{\emph{Proof of Claim.}}
        By Lemma~\ref{lem: stop value} and Theroem~\ref{thm: stop dependency}, we have
        \begin{align}
           \Ls_b &= n + c_1, \label{case1b1:eq1}\\
           \Rs_b &= -n + c_2, \label{case1b1:eq2}\\
           \Rs_\sigma  &= n + c_3, \label{case1b1:eq3}
        \end{align} where $c_1,c_2$ and $c_3$ are constants with respect to $n$. 
        Since $R_b\in\DT$, Equations~\eqref{case1b1:eq1}-\eqref{case1b1:eq2} and Lemma~\ref{lem: temp of tents}  together imply $m_b= \left(\Ls_b+\Rs_b\right)/2 = \left(c_1+c_2\right)/2$. 
        Similarly, Equation~\eqref{case1b1:eq3} imply $m_\sigma=\Rs_\sigma = n+c_3$ as $R_\sigma\in\LST$. Therefore, since we are taking $n$ to be large, we have 
        \begin{equation}\label{case1b1:eq4}
            m_b\leq m_\sigma.
        \end{equation}
        
        
        \begin{figure}[ht!]
		\centering
	 	\begin{subfigure}[b]{0.45\textwidth} 
                \begin{tikzpicture}[>=stealth, scale=0.6, dot/.style = {circle, fill, minimum size=#1, inner sep=0pt, outer sep = 0pt}, dot/.default = 6pt]  
                    \draw [<->] (-1,0) -- (7,0) node [at end, right] { };
                    \node[dot, fill=green!50!black, scale=0.3] (l1) at (0,0) { };
                    \node[dot, scale=0.3, fill=green!50!black] (r1) at (2,0) { };
                    \node[dot, scale=0.3, fill=red] (l2) at (1,0) { };
                    \node[dot, scale=0.3, fill=red] (r2) at (6,0) { };
                    \draw[green!50!black, very thick] (l1)--+(2,2) node[dot, scale=0.3] (C) { };
                    \draw[dotted] (7,2)--(0,2) node[at end, label=left:{$p=t_\sigma$}] (t1) { };
                    \draw[green!50!black, very thick] (C)--+(0,2.5) node[anchor=east] (D) {$\bm{R_\sigma}$};
                    \draw[green!50!black] (r1)--(C);
                    \draw[red,very thick] (l2)--+(2.5,2.5) node[dot, scale=0.3] (G) { };
                    \draw[red] (r2)--(G);
                    \draw[red, very thick] (G)--+(0,2) node[anchor=west] (H) {$\bm{R_b}$};
                    \draw[dotted] (6,2.5)--(0,2.5) node[at end, label=left:{$p=t_b$}] (tb) { };
		    \end{tikzpicture}
                \caption{$\LW_b$ and $\LW_\sigma$, when $t_b\ge t_\sigma$.}
                \label{fig: Thermograph of Right options in case 1b2}
	 	\end{subfigure}%
			\hfill
            \begin{subfigure}[b]{0.45\textwidth}
                \begin{tikzpicture}[>=stealth, scale=0.6, dot/.style = {circle, fill, minimum size=#1, inner sep=0pt, outer sep = 0pt}, dot/.default = 6pt]  
                    \draw [<->] (-1,0) -- (6.5,0) node [at end, right] { };
                    \node[dot, fill=green!50!black, scale=0.3] (l1) at (0,0) { };
                    \node[dot, scale=0.3, fill=green!50!black] (r1) at (3,0) { };
                    \node[dot, scale=0.3, fill=red] (l2) at (2,0) { };
                    \node[dot, scale=0.3, fill=red] (r2) at (6,0) { };
                    \draw[green!50!black, very thick] (l1)--+(3,3) node[dot, scale=0.3] (C) { };
                    \draw[green!50!black, very thick] (C)--+(0,1.5) node[anchor=east] (D) {$\bm{R_\sigma}$};
                    \draw[green!50!black] (r1)--(C);
                    \draw[red, very thick] (l2)--+(2,2) node[dot, scale=0.3] (G) { };
                    \draw[red] (r2)--(G);
                    \draw[red, very thick] (G)--+(0,2.5) node[anchor=west] (H) {$\bm{R_b}$};
                    \draw[dotted] (6,2)--(0,2) node[at end, label=left:{$p=t_b$}] (tb) { };
                    \draw[dotted] (6,3)--(0,3) node[at end, label=left:{$p=t_\sigma$}] (t1) { };
		    \end{tikzpicture}
                \caption{$\LW(R_b)$ and $\LW(R_\sigma)$, when $t(R_b)<t(R_\sigma)$.}
                \label{fig: Thermograph of Right options in case 1b3}
            \end{subfigure}
            \caption{Thermographs of Right options of $G$, in case~(1B).}
            \label{fig: Thermograph of Right options in case 1b}
	\end{figure}
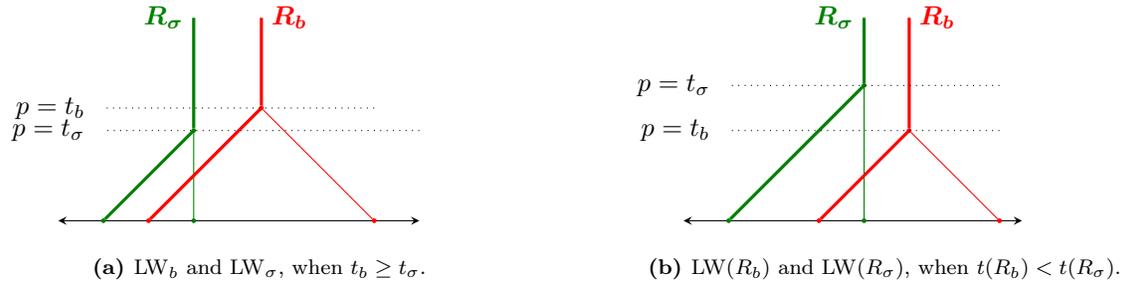

       Since both small left walls, \(\lw_b\) and \(\lw_\sigma\), have slope \(-1\), and \(\Ls_b \leq \Ls_\sigma\) as well as \(m_b \leq m_\sigma\), the proof of \(\LW_b \leq \LW_\sigma\) follows similarly to Case~(1A). Therefore, we omit the details. 
       This argument does not depend on other properties of the thermographs, such as the relation between temperatures; see Figure~\ref{fig: Thermograph of Right options in case 1b}.



        We omit the analysis of $R_b\in \DT$ and $R_\delta \in \DT$, as it parallels that of $R_b\in\LST$ with $R_\sigma\in\LST$ in Case~(1A), with the only distinction being the argument that $m(R_b)\leq m(R_\delta)$ due to the shape of the thermographs and the monotonicity of the stops. 
        This completes the proof in this case.  
        \end{enumerate}
        
        \item[\bf{(2)}] $\bm{1<a/b<\phi\!:}$ 
        In this case, the wealth ratios \(w_i=(a-i)/b\) for the Right options fall within the interval \((0, \phi)\). By the induction hypothesis, $R_i\in \DT\cup\RST$. Define \(\beta\) as the smallest right removal index for which $R_i\in\RST$, i.e.,  
        \[
        \beta = \min\Set{i \geq 1 \SetSymbol w_i < \phi^{-1}}.
        \]  
        Observe that \(\beta \leq b\), since \((a-b)/b = a/b - 1 < \phi^{-1}\). Consequently, \(R_\delta\in \DT\) for \(\delta < \beta\), while $R_\sigma\in\RST$, if $\sigma\ge \beta$. Hence, we must compare the large left wall of $R_b\in\RST$ with those of both $R_\delta\in\DT$ and $R_\sigma\in\RST$.
        
        We begin by verifying that $\LW_b\le \LW_\delta$, for large $n$. 

         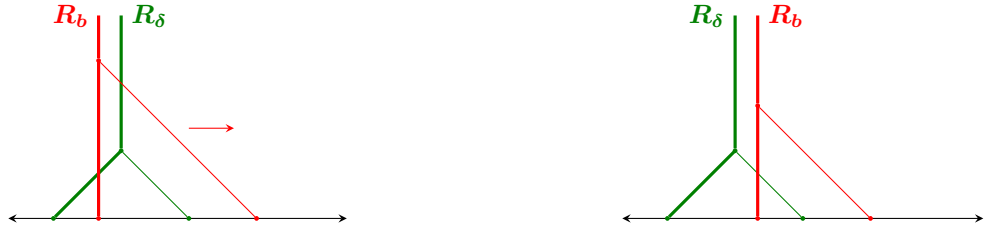
\begin{figure}[ht!]
	 	\centering
	  	\begin{subfigure}[b]{0.45\textwidth}
                \centering
                 \begin{tikzpicture}[>=stealth, scale=0.6, dot/.style = {circle, fill, minimum size=#1, inner sep=0pt, outer sep = 0pt}, dot/.default = 6pt]  
                     \draw [<->] (-1,0) -- (6.5,0) node [at end, right] { };
                     \node[dot, fill=green!50!black, scale=0.3] (l1) at (0,0) { };
                     \node[dot, scale=0.3, fill=green!50!black] (r1) at (3,0) { };
                     \node[dot, scale=0.3, fill=red] (l2) at (1,0) { };
                     \node[dot, scale=0.3, fill=red] (r2) at (4.5,0) { };
                     \draw[green!50!black, very thick] (l1)--+(1.5,1.5) node[dot, scale=0.3] (C) { };
                     \draw[green!50!black] (r1)--(C);
                     \draw[green!50!black, very thick] (C)--+(0,3) node[anchor=west] (D) {$\bm{R_\delta}$};
                     \draw[red, very thick] (l2)--+(0,3.5) node[dot, scale=0.3] (G) { };
                     \draw[red] (r2)--(G);
                     \draw[red, very thick] (G)--+(0,1) node[anchor=east] (H) {$\bm{R_b}$};
                     \draw[red,->] (3,2)--(4,2);
                    
                    ;
	 	    \end{tikzpicture}
                \caption{The case $m_b>m_\delta$ does not appear for large $n$.}
                 \label{fig: Thermograph of Right options in case 2.1}
	  	\end{subfigure}%
	 		\hfill
             \begin{subfigure}[b]{0.45\textwidth} 
                \centering
                \begin{tikzpicture}[>=stealth, scale=0.6, dot/.style = {circle, fill, minimum size=#1, inner sep=0pt, outer sep = 0pt}, dot/.default = 6pt]  
                     \draw [<->] (-1,0) -- (7,0) node [at end, right] { };
                     \node[dot, fill=green!50!black, scale=0.3] (l1) at (0,0) { };
                     \node[dot, scale=0.3, fill=green!50!black] (r1) at (3,0) { };
                     \node[dot, scale=0.3, fill=red] (l2) at (2,0) { };
                     \node[dot, scale=0.3, fill=red] (r2) at (4.5,0) { };
                     \draw[green!50!black, very thick] (l1)--+(1.5,1.5) node[dot, scale=0.3] (C) { };
                     \draw[green!50!black] (r1)--(C);
                     \draw[green!50!black, very thick] (C)--+(0,3) node[anchor=east] (D) {$\bm{R_\delta}$};
                     \draw[red, very thick] (l2)--+(0,2.5) node[dot, scale=0.3] (G) { };
                     \draw[red] (r2)--(G);
                     \draw[red, very thick] (G)--+(0,2) node[anchor=west] (H) {$\bm{R_b}$};
                                       
	 	    \end{tikzpicture}
                 \caption{The case $m_b\le m_\delta$ appears for large $n$.}
                 \label{fig: Thermograph of Right options in case 2.2}
	  	\end{subfigure}
            \caption{Thermographs of the Right options $R_b$ and $R_\delta$ of $G$ in case~(2).}
             \label{fig: Thermograph of Right options in case (2)}
	 \end{figure}

        Similar to the claim in case (1B), for large $n$, we get $m_b\leq m_\delta$. We indicate with a red arrow in Figure~\ref{fig: Thermograph of Right options in case 2.1} that $\T(R_b)$ shifts to the right with increasing $n$, which instead creates a situation as in (b). Indeed, 
        since $R_b\in \RST$, $\LW_b$ is a vertical line, and thus $\LW_\delta\ge \LW_b$ by stop monotonicity and $m_b\leq m_\delta$ (see Figure~\ref{fig: Thermograph of Right options in case 2.2}).

        Next we compare $R_b\in\RST$ with $R_\sigma\in \RST$. 
        In this case, the Left walls of both $R_b$ and $R_\sigma$ are vertical lines determined by their Left stops, i.e., for all $p\ge 0$,
            $\Ls_b(p)=\Ls_b$ and
            $\Ls_\sigma(p)=\Ls_\sigma$. 
        By Lemma~\ref{lem: osm (option stop monoto)}, we have $\Ls_b\leq \Ls_\sigma$. Hence, $\LW_b\le \LW_\sigma$. 


        \item[\bf{(3)}] $\bm{a/b=1\!:}$ In this case, $R_b=(n-a;0,b) = -(n-a)$. Thus, $R_b\in \M$. Now, by Lemma~\ref{lem: osm (option stop monoto)}, $R_b$ is the Right option of $G$ with the smallest Right stop and hence, by applying Corollary~\ref{cor: MvsO mast vs options}, $\RW(G)$ does not depend on any other Right option than $R_b$.
        
    \end{enumerate}
Thus the {\sc Robin Hood} thermographs are the same as those of {\sc Little John}.
\end{proof}

The claim in Case~(1B) in the proof of Theorem~\ref{thm:RHthermographs}, might not hold for smaller heap sizes, and we explain what can happen in terms of thermographs in Figure~\ref{fig: Right slant of G in case 1b1}; our ``tent structures'' may not survive for small heap sizes.

        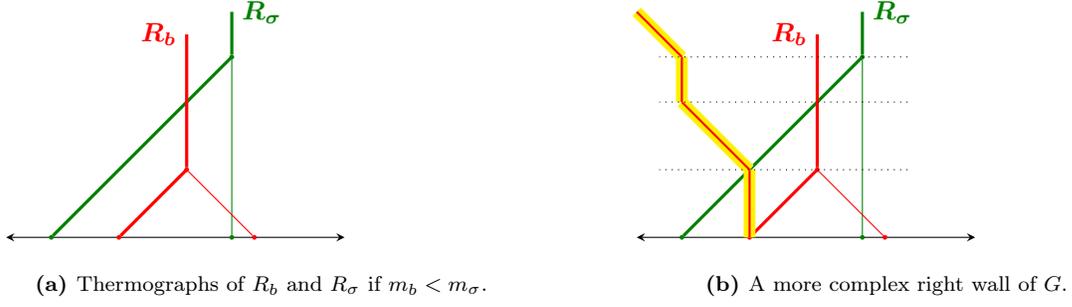
\begin{figure}[ht!]
		\centering
	 	\begin{subfigure}[b]{0.45\textwidth} 
                \begin{tikzpicture}[>=stealth, scale=0.6, dot/.style = {circle, fill, minimum size=#1, inner sep=0pt, outer sep = 0pt}, dot/.default = 6pt]  
                    \draw [<->] (-1,0) -- (6.5,0) node [at end, right] { };
                    \node[dot, fill=green!50!black, scale=0.3] (l1) at (0,0) { };
                    \node[dot, scale=0.3, fill=green!50!black] (r1) at (4,0) { };
                    \node[dot, scale=0.3, fill=red] (l2) at (1.5,0) { };
                    \node[dot, scale=0.3, fill=red] (r2) at (4.5,0) { };
                    \draw[green!50!black, very thick] (l1)--+(4,4) node[dot, scale=0.3] (C) { };
                    \draw[green!50!black] (r1)-- (C);
                    \draw[green!50!black, very thick] (C)--+(0,1) node[anchor=west] (D) {$\bm{R_\sigma}$};
                    \draw[red, very thick] (l2)--+(1.5,1.5) node[dot, scale=0.3] (G) { };
                    \draw[red] (r2)--(G);
                    \draw[red, very thick] (G)--+(0,3) node[anchor=east] (H) {$\bm{R_b}$};
                    
		    \end{tikzpicture}
                \caption{Thermographs of $R_b$ and $R_\sigma$ if $m_b< m_\sigma$.}
                \label{fig: Thermograph of Right options in case 1b1}
	 	\end{subfigure}%
			\hfill
            \begin{subfigure}[b]{0.45\textwidth}
                \begin{tikzpicture}[>=stealth, scale=0.6, dot/.style = {circle, fill, minimum size=#1, inner sep=0pt, outer sep = 0pt}, dot/.default = 6pt]  
                   \draw [<->] (-1,0) -- (6.5,0) node [at end, right] { };
                    \node[dot, fill=green!50!black, scale=0.3] (l1) at (0,0) { };
                    \node[dot, scale=0.3, fill=green!50!black] (r1) at (4,0) { };
                    \node[dot, scale=0.3, fill=red] (l2) at (1.5,0) { };
                    \node[dot, scale=0.3, fill=red] (r2) at (4.5,0) { };
                    \draw[green!50!black, very thick] (l1)--+(4,4) node[dot, scale=0.3] (C) { };
                    \draw[green!50!black] (r1)-- (C);
                    \draw[green!50!black, very thick] (C)--+(0,1) node[anchor=west] (D) {$\bm{R_\sigma}$};
                    \draw[red, very thick] (l2)--+(1.5,1.5) node[dot, scale=0.3] (G) { };
                    \draw[red] (r2)--(G);
                    \draw[red, very thick] (G)--+(0,3) node[anchor=east] (H) {$\bm{R_b}$};
                    \draw[red,thick, preaction={draw,yellow,-,double=yellow, double distance=1mm}] (l2.center) -- +(0,1.5) node (e1) { };
                    \draw[red,thick, preaction={draw,yellow,-,double=yellow, double distance=1mm}] (e1.center)-- +(-1.5,1.5) node (e2) { };
                    \draw[red,thick, preaction={draw,yellow,-,double=yellow, double distance=1mm}] (e2.center)-- +(0,1) node (e3) { };
                    \draw[red,thick, preaction={draw,yellow,-,double=yellow, double distance=1mm}] (e3.center)-- +(-1,1);
                    \draw[dotted] (5,1.5) -- (-0.5,1.5);
                    \draw[dotted] (5,4) -- (-0.5,4);
                    \draw[dotted] (5,3) -- (-0.5,3);
		    \end{tikzpicture}
                \caption{A more complex right wall of $G$.}
                \label{fig: Right slant of G in case 1b1}
            \end{subfigure}
            \caption{Thermographs of Right options, as in case (1B), for smaller heap sizes.}
            \label{fig: Thermograph of Right options in case 1b smaller heap sizes}
	\end{figure}

We can now compute the temperatures and mean values of {\sc Robin Hood} as stated in the main theorem from the Introduction, Theorem~\ref{thm: main theorem}.\\

\noindent{\bf Theorem~1.1} (Main Theorem) 
For fixed positive integers $a$ and $b$, let $G_n = (n;a,b)$ be instances of {\sc Robin Hood}. Let $(U_i)_{i\geq 0}$ be the unique sequence of positive integers such that
\begin{enumerate}
    \item $U_0\geq U_1$,
    \item $U_{k+2}=U_{k+1}+U_k \text{ for all } k\geq 0$, and
    \item for some $\mu\geq0$, $U_\mu=\min\Set{a,b}$ and $U_{\mu+1}=\max\Set{a,b}$. 
\end{enumerate} 
     For all sufficiently large heap sizes $n$, the temperature of $G_n$ is\\   $$t(G_n)=
  \begin{cases} 
        b-U_0 & \text{ if }~ \frac{a}{b} < \phi^{-1},\\
        n-a + \frac{U_0-b}{2} & \text{ if }~ \phi^{-1}<\frac{a}{b} <1,\\
        n-a & \text{ if }~ \frac{a}{b} = 1,\\
        n-b+\frac{U_0-a}{2} & \text{ if }~ 1<\frac{a}{b} < \phi,\\
        a-U_0 & \text{ if }~ \phi<\frac{a}{b},\\
  \end{cases}$$
  and the mean value  is\\   $$m(G_n)=
  \begin{cases} 
        a+b-n-U_0 & \text{ if }~ \frac{a}{b} < \phi^{-1},\\
        \frac{U_0-b}{2} & \text{ if }~ \phi^{-1}<\frac{a}{b} <1,\\
        0 & \text{ if }~ \frac{a}{b} = 1,\\
        \frac{a-U_0}{2} & \text{ if }~ 1<\frac{a}{b} < \phi,\\
        n-(a+b)+U_0 & \text{ if }~ \phi<\frac{a}{b}.\\
  \end{cases}$$

\begin{proof}
Let \( G_n^* \) represent the Little John game associated with \( G_n \), i.e., \( G_n^* = (n; a, b)^* \). According to Theorem~\ref{thm:RHthermographs}, for sufficiently large \( n \), \( \T(G_n) \) equals \( \T(G_n^*) \), implying:
\[
t(G_n) = t(G_n^*) \quad \text{and} \quad m(G_n) = m(G_n^*).
\]
Thus, the problem reduces to determining the temperature and mean value of \( G_n^* \). As shown in Theorem~\ref{thm: LJthermograph}, \( \T(G_n^*) \) depends on the wealth ratio \( a/b \) for sufficiently large \( n \). Therefore, we consider different cases based on the wealth ratio \( a/b \) and take \( n \) to be large enough that both Theorems~\ref{thm: LJthermograph} and \ref{thm:RHthermographs} are applicable.

\begin{enumerate}
    \item[\bf{(1)}]  \bm{$ a/b < \phi^{-1}\!: $} In this case, by Theorem~\ref{thm: LJthermograph}, \( G_n^* \in \RST \), and by Lemma~\ref{lem: stop value}, the stops are \( s(G_n^*) = (a + b - n - U_0, \; a - n) \). Hence, we get:
    \[
    m(G_n^*) = a + b - n - U_0 \quad \text{and} \quad t(G_n^*) = b - U_0,
    \]
    as per Lemma~\ref{lem: temp of tents}. A similar argument applies if \( a/b > \phi \).
    
    \item[\bf{(2)}]  \bm{$ \phi^{-1} < a/b < 1\!: $} In this case, by Theorem~\ref{thm: LJthermograph}, \( G_n^* \in \DT \), and by Lemma~\ref{lem: stop value}, the stops are \( s(G_n^*) = (n - (a + b) + U_0, \; a - n) \). Therefore, we have:
    \[
    m(G_n^*) = \frac{U_0 - b}{2} \quad \text{and} \quad t(G_n^*) = n - a + \frac{U_0 - b}{2}
    \]
    as shown in Lemma~\ref{lem: temp of tents}. A symmetric argument holds if \( 1 < a/b < \phi \).
    
    \item[\bf{(3)}]  \bm{$ a/b = 1\!: $} In this case, by Theorem~\ref{thm: LJthermograph}, \( G_n^* \in \DT \), and by Lemma~\ref{lem: stop value}, the stops are \( s(G_n^*) = (n - a, \; a - n) \). The desired result follows by applying Lemma~\ref{lem: temp of tents}.
\end{enumerate}

This concludes the proof.
\end{proof}

\section*{Open problems}
\begin{enumerate}[ labelindent=0pt]
    \item Solve the `middle region' of {\sc Robin Hood}, in terms of thermographs, whenever  $0<\min\{a,b\}< n\le a+b$. 
    \item Find explicit bounds on ``sufficiently large'' heap sizes for the main theorem to apply.
    \item Study the Canonical Forms of {\sc Robin Hood}, and in cases where this is hard, study instead the so-called Reduced Canonical Form \cite{S2013}.
    \item Study more instances of {\sc Wealth Nim}.
    \item Study the outcomes of a ruleset where there is only one global pair of wealths for the entire compound, governing how many tokens may be removed from any single heap. The reduction of the opponent's wealth would be the same as here. For example if the wealth pair is $(a,b)=(2,3)$ and there are two heaps of sizes one and two, then, if Left starts, she can win if she removes one token from the second heap and reduce Right's wealth by one rupee. However, if she plays the Little John move and reduces Right's wealth by two rupees, and removes the second heap, then she loses. In this perspective, {\sc Robin Hood} is a local variation, where each heap has an individual wealth pair. Compare with two papers on {\em local} vs. {\em global} {\sc Fibonacci Nim} \cites{LR-S2015,LR-S2018}.
    \item Given a number of heaps, and a global wealth for each player, solve the simultaneous play `optimal' assignments of wealth per heap, in the sense of Colonel Blotto \cite{EB1921}. In cases where hotstrat applies, the main results of this paper should guide such assignments, and otherwise one would need better understanding of (reduced) canonical forms.

\end{enumerate}

\section*{Acknowledgments}
We would like to express our heartfelt gratitude to Prof. Carlos Pereira dos Santos for his invaluable assistance and insightful discussions; in particular concerning Theorem~\ref{thm: RH positions}. 

\bibliographystyle{plain}

\end{document}